\documentclass[final,onefignum,onetabnum]{siamart190516}

\usepackage{amsmath,amsfonts,amssymb}
\usepackage{color}
\usepackage[]{graphics}
\usepackage{verbatim}
\usepackage{eepic,epic}
\usepackage{epsfig,subfigure,epstopdf}
\usepackage{hyperref}
\usepackage{ulem}
\usepackage{algorithmic}

\newtheorem{Lemma}{Lemma}[section]
\newtheorem{example}{Example}[section]
\newtheorem{remark}{Remark}[section]
\newtheorem{Corollary}{Corollary}[section]
\numberwithin{equation}{section}

\newcommand{\al}{\alpha}

\def\Dal{{\partial_t^\al}}

\def\Om{\Omega}
\def\II{(\Om)}

\def\d{{\mathrm d}}
\begin{document}

\title{Time-Fractional Allen-Cahn Equations: Analysis and Numerical Methods\thanks{This research is supported in part by the ARO MURI Grant W911NF-15-1-0562 and NSF DMS-1719699.  The work of J. Yang is partially supported by supported  by National Science Foundation of China NSFC-11871264. The work of Z. Zhou is supported by the start-up grant from the Hong Kong Polytechnic University and Hong Kong RGC
grant No. 25300818.}}

\author{Qiang Du\thanks{Department of Applied Physics and Applied Mathematics,
Columbia University, New York, NY 10027, {\tt qd2125@columbia.edu}}
\and
Jiang Yang\thanks{Department of Mathematics, Southern University of Science and Technology,
Shenzhen, China, {\tt yangj7@sustech.edu.cn}}
\and
Zhi Zhou\thanks{Department of Applied Mathematics, The Hong Kong Polytechnic University, Kowloon, Hong Kong
{\tt zhizhou@polyu.edu.hk, zhizhou0125@gmail.com}}
}

\maketitle

\begin{abstract}
In this work, we consider a time-fractional Allen-Cahn equation, where the conventional first order time derivative is replaced by a Caputo fractional derivative
with order $\alpha\in(0,1)$.
First, the well-posedness and (limited) smoothing property are systematically analyzed, by using the maximal $L^p$ regularity of fractional evolution equations
and the fractional Gr\"onwall's inequality. We  also show the maximum principle like their conventional local-in-time counterpart. Precisely, the time-fractional equation preserves the
property that the solution only takes value  between the wells of the double-well potential when the initial
data does the same.
Second, after discretizing the fractional derivative by backward Euler convolution quadrature,
we develop several unconditionally solvable and stable time stepping schemes, i.e., convex splitting scheme, weighted convex splitting scheme and linear weighted stabilized scheme.
Meanwhile, we study the discrete energy dissipation property (in a weighted average sense), which is important for gradient flow type models, for the two weighted schemes.
Finally, by using a discrete version of fractional Gr\"onwall's inequality
and maximal $\ell^p$ regularity, we prove that the convergence rates of those time-stepping schemes are $O(\tau^\alpha)$ without any extra regularity assumption on the solution.
We also present extensive numerical results to support our theoretical findings and to offer new
insight on the time-fractional Allen-Cahn dynamics.
\end{abstract}



\section{Introduction}\label{sec:intro}
Classical phase-field models are diffuse interface models that have found numerous applications
in diverse research areas, e.g., hydrodynamics \cite{Anderson:1998, Liu:2003, Qian:2003},
material sciences \cite{Allen:1979, Cahn:1958}, biology \cite{Du:2004, Shao:2010, Wise:2008}
and image processing \cite{Paragios:2001, Xu:1998}, to name just a few.
Recently, there are also many studies on nonlocal phase field models involving spatially nonlocal interactions
\cite{Akagi:2016,Antil:2017,Bates:2006,Bates:2009,DLLZ:2018,DLLZ:2019,Guan:2014,Gui:2015,Song:2017,Valdinoci:2013}, 
see \cite{Du:2019,DF:2019} for more extensive reviews of the literature. Historically,
nonlocal interactions in phase field models expressed mathematically in terms of integral operators have been noted
in the work of van der Waals \cite{Waals:1894}, see discussions made in \cite{Pismen:2001}.  
In fact, one can deduce the usual differential equation form of the local phase field energy from the nonlocal version via the so-called 
Landau expansion \cite{Landau:2013}, under the usual assumption on the smooth and slowly varying nature of the phase field variables.
Meanwhile,  it has been reported that the presence of nonlocal operators in  either time \cite{ChenZhang:2018, Liu:2018, Tang:2018} or space \cite{Ainthworth:2017, Caffarelli:2010, DuYang:2016, Song:2017} in phase field equations  may change diffusive dynamics significantly.

In this paper, we shall focus on the following time-fractional Allen-Cahn equation \cite{Liu:2018, Tang:2018}: 
\begin{align}\label{eqn:phase-field}
\left\{\begin{aligned}
\partial_t^\alpha u - \kappa^2 \Delta u &= -F'(u)=:f(u)
&&\mbox{in}\,\,\,\Omega\times(0,T) ,\\
u(x,t)&=0  &&\mbox{in}\,\,\,\partial\Omega\times(0,T),\\
u(x,0)&=u_0  &&\mbox{in}\,\,\,\Omega,\\
\end{aligned}\right.
\end{align}
where the function $F(u)$ is bistable, e.g.,$$F(s)=\frac14(1-s^2)^2.$$
Here $\Omega$ is a smooth domain of $\mathbb{R}^d$ with $d=1,2,3$ with boundary $\partial \Omega$.
The operator $\partial_t^\alpha $ denotes the Caputo-type fractional derivative of order $\alpha\in(0,1)$ in time,
which is a typical example of nonlocal operators and defined
by \cite[p. 70]{KilbasSrivastavaTrujillo:2006}
\begin{align}\label{eqn:Caputo}
   \Dal u(x,t)= \frac{1}{\Gamma(1-\alpha)}\int_0^t(t-s)^{-\alpha}\partial_s u (x,s)\d s.
\end{align}

Time-fractional PDEs have attracted some attention in the modeling anomalous  diffusion,
e.g., protein diffusion within cells \cite{golding2006physical}, contaminant transport in groundwater \cite{kirchner2000fractal},
and thermal diffusion in media with fractal geometry \cite{nigmatullin1986realization}.
 For the phase-field model \eqref{eqn:phase-field}, some numerical studies have been presented
\cite{Liu:2018}. It was reported there that  the order $\alpha$ affects significantly the relaxation time to the equilibrium
and the sharpness of the interface. Their simulations were based on
a fully discrete scheme for  \eqref{eqn:phase-field} that uses a
 piecewise linear interpolation to discretize the fractional derivative,
and the Fourier spectral method in space.
Meanwhile, a fast algorithm was proposed to reduce the computational complexity and the storage requirement.
However, they did not provide any stability and convergence analysis.
 Very recently, Chen et al. numerically studied
a time-fractional molecular beam epitaxy (MBE) model using a similar fully discrete scheme  \cite{ChenZhang:2018}.
They observed that both of energy decay rate and coarsening rate satisfy a power-law
determined by the fractional order. The first theoretical work regarding the energy stability for the model \eqref{eqn:phase-field} was given by Tang et al. in \cite{Tang:2018}, where
they showed that the energy at a latter time must be bounded by the initial energy, i.e.,
\begin{align}\label{eqn:energy-bound}
E(u(t)) \le E(u_0),\quad \text{where}~~  E(u) = \int_\Omega \frac{\kappa^2}2 |\nabla u(x) |^2   + F(u(x)) \,dx.
\end{align}
In the discrete level, they proposed a stabilized time stepping scheme which keeps this kind of stability.
Note that the original local-in-time integer-order Allen-Cahn is a gradient flow, so the free energy $E(u(t))$ is decreasing.
Nevertheless, the energy dissipation has not been discussed, except for the boundedness of the energy
 \eqref{eqn:energy-bound},
at neither continuous nor discrete level.  Moreover,
none of the aforementioned works study the well-posedness and the regularity.
The sharp error estimate of those time stepping schemes is also unavailable in the literature,
due to the unknown smoothing property of the solution.
Let us also add that a related time-fractional Cahn-Hilliard model was also {numerically} studied in \cite{Liu:2018, Tang:2018}, where the energy dissipation rate
is reported to satisfy a power-law $t^{-\alpha/3}$.
Indeed, in-depth studies of these time-fractional phase field type models remain fairly scarce.

The goal of our work is to provide more extensive studies of nonlocal-in-time phase field models, and to develop and analyze stable numerical schemes.
We shall present a rigorous analysis of the model \eqref{eqn:phase-field} with smooth initial data,
provide sharp regularity estimate, investigate time stepping
schemes, establish the optimal error estimates, and study the discrete energy dissipation law.
Compared with the classical phase-field equation,
the analysis of the nonlocal one will be more challenging, since the nonlocal
(differential) operator does not preserve some of the rules enjoyed by their local counterpart,
e.g., the product rule and the chain rule.

The first contribution of the paper is to study the well-posedness and regularity of the model \eqref{eqn:phase-field}, in Section \ref{sec:PDE}.
Nonlinear time-fractional diffusion equations with globally Lipschitz continuous potential term has already been investigated in \cite{JinLiZhou:nonlinear}.
However, the loss of global Lipschitz continuity in the time-fractional phase-field equation \eqref{eqn:phase-field} will result in additional troubles.
To overcome this difficulty, we apply the energy argument, the maximal $L^p$ regularity and
fractional Gr\"onwall's inequality, to prove the existence and uniqueness of a weak solution, which satisfies (Theorem \ref{thm:exists})
\begin{equation*}
\|  \partial_t^\alpha u  \|_{L^p(0,T;L^2\II)} +  \|  \Delta u  \|_{L^p(0,T;L^2\II)} \le c,\qquad \text{for any} ~~p\in[2,2/\alpha).
\end{equation*}
provided that $u_0\in H_0^1\II$. Further, the regularity will be improved that $u\in L^\infty((0,T)\times\Omega)$ under the assumption that $u_0\in H^2\II\cap H_0^1\II$.
These results combining with the argument in \cite[Theorem 3.1]{JinLiZhou:nonlinear} result in the sharp regularity estimates.
One application of those regularity estimates is to prove the maximum bound principle
in the sense that if $|u_0|\le 1$ then the solution $u$ shares the same pointwise maximum norm  bound at all time.
This is a  property enjoyed by the conventional local-in-time Allen Cahn equation.
Note that such a  principle has been shown in \cite{Tang:2018}, but it was
under certain a priori regularity assumptions that have not been rigorously confirmed so far.

Our second contribution is to develop some unconditionally solvable and stable schemes in Section \ref{sec:scheme}.
The first approach is motivated by the classical convex splitting method (CS). This scheme satisfies
the discrete maximum norm bound principle. Moreover, it satisfies the energy stability, i.e.,
\begin{equation*}
E(u_n) \le E(u_0) \qquad \text{for all}\quad  n \ge 1.
\end{equation*}
However, the energy dissipation law of the CS scheme is hard to establish due to the nonlocal effect of the fractional derivative.
This motivates us to develop two other schemes, named as
weighted convex splitting (WCS) scheme and linear weighted stabilized (LWS) scheme. Both schemes satisfy the discrete  maximum bound principle
as well as a weighted energy dissipated law, i.e.,
\begin{equation*}
 E(u_n)\le E(u_{n,\alpha}),\qquad \text{for all}\quad n\ge 1,
\end{equation*}
where $u_{n,\alpha}$ is a convex combination of $u_0,u_1,\ldots,u_{n-1}$. In cases where the free energy adopted under
consideration is convex (e.g., the linear problem or time-fractional
Allen-Cahn with $|u_n|\ge \frac{\sqrt3}{3}$), the weighted energy stability indicates the fractional energy dissipation law as
$\bar\partial_\tau^\alpha E(u_n)\le 0$.
This is consistent with the energy decay property of classical gradient flow ($\alpha=1$),
$\bar\partial_\tau E(u_n)\le 0.$
See more discussion on Remark \ref{rem:egy-disc}.

As additional analytical studies, we analyze the convergence of those time-stepping schemes in Section \ref{sec:error}.
The pointwise-in-time errors are derived
by applying the regularity estimates in Section \ref{sec:PDE}, the discrete fractional Gr\"onwall's inequality established in \cite{JinLiZhou:nonlinear}
and discrete maximal $\ell^p$ regularity derived in \cite{JinLiZhou:max-reg},
as well as some novel stability estimates proved in Section \ref{sec:error}.
In particular, under the assumption that $u_0\in H^2\II\cap H_0^1\II$  and $|u_0(x)|\le 1$, we prove that
all those three time stepping solutions satisfy, under no additional regularity assumptions, that
\begin{equation*}
\max_{1\le n\le N}  \|u(t_n)-u_n\|_{L^2(\Omega)}\le c\tau^\alpha .
\end{equation*}
where $c$ denotes a generic constant depending on the $\alpha$, $u_0$, $T$ and $\kappa$, but always independent of $\tau$ and any smoothness of $u$.

Finally, in Section \ref{sec:numerics}, we present some numerical experiments to confirm the theoretical findings and to offer new
insight on the time-fractional Allen-Cahn dynamics.  Numerically, we observe that the relaxation time to the steady state gets longer for smaller $\alpha$,
and the time-fractional Allen--Cahn equation converge to some slow mean curvature flow.
Those phenomena await further theoretical understanding in the future.

Throughout this paper, the notation $c$ denotes a generic constant, which may depends on $\alpha, \kappa, T, u_0$
but is always independent of $u$ and step size $\tau$.

\bigskip
\section{Solution theory of the time-fractional Allen-Cahn equations}\label{sec:PDE}

In this section, we shall study the well-posedness of the time-fractional Allen-Cahn equation \eqref{eqn:phase-field} and prove the sharp regularity.
The argument for nonlinear fractional diffusion equation with globally lipschitz continuous potential has been investigated in \cite{JinLiZhou:nonlinear}.
To overcome the lack of global Lipschitz property of the nonlinear potential $f(s)$, it is important to
show $u\in L^\infty((0,T)\times\Omega)$ in the first place. To this end, we shall use the following fractional Gr\"onwall's inequality.

\begin{Lemma}\label{Frac-Gronwall}
Suppose that $y$ is nonnegative and $y$ satisfies the inequality
\begin{equation*}
\partial_t^\alpha y(t) \le \beta y(t) + \sigma(t),
\end{equation*}
where the function $\sigma\in L^\infty(0,T)$ and the constant $\beta>0$. Then
\begin{equation*}
y(t) \le c_T \big(y(0) + \| \sigma \|_{L^\infty(0,T)}\big).
\end{equation*}
where the constant $C_T$ where the constant $c$ is independent of $\sigma$ and $y$,
but may depend on $\alpha$,  $\beta$ and $T$.
\end{Lemma}

\begin{proof}
We define an axillary function $w(t)$ such that
\begin{equation*}
w(t) = E_{\alpha,1}(\beta t^\alpha)y(0) + \int_0^t s^{\alpha-1} E_{\alpha,\alpha}(\beta s^\alpha) \sigma(t-s) \,ds
\end{equation*}
where the $E_{a,b}(z)$ denotes the Mittag-Leffler function \cite{KilbasSrivastavaTrujillo:2006}.
Then the function $w(t)$ satisfies the fractional initial value problem
\begin{equation*}
\partial_t^\alpha w(t) = \beta w(t) + \sigma(t),\quad \text{with}~~ w(0)=y(0),
\end{equation*}
Using the positivity and boundedness of the Mittag-Leffler function, we obtain that
\begin{equation*}
w(t) \le  c_1 y(0) + c_2 \int_0^t s^{\alpha-1}  \sigma(t-s) \,ds \le c_T( y(0) +  \| \sigma \|_{L^\infty(0,T)}).
\end{equation*}
and the desired result follows from the comparison principle, i.e., $y(t) \le w(t)$.
\end{proof}

In order to study the well-posedness and regularity, we shall use the Bochner space. For a Banach space $B$, we define
\begin{equation*}
  L^r(0,T;B) = \{u(t)\in B \mbox{ for a.e. } t\in (0,T) \mbox{ and } \|u\|_{L^r(0,T;B)}<\infty\},
\end{equation*}
for any $r\geq 1$, and the norm $\|\cdot\|_{L^r(0,T;B)}$ is defined by
\begin{equation*}
  \|u\|_{L^r(0,T;B)} = \left\{\begin{aligned}\left(\int_0^T\|u(t)\|_B^rdt\right)^{1/r}, &\quad r\in [1,\infty),\\
       \mathrm{esssup}_{t\in(0,T)}\|u(t)\|_B, &\quad r= \infty.
       \end{aligned}\right.
\end{equation*}
Besides, we shall use extensively Bochner-Sobolev spaces $W^{s,p}(0,T;B)$. For any $s\ge 0$ and
$1\le p < \infty$, we denote by $W^{s,p}(0,T;B)$ the space of functions $v:(0,T)\rightarrow B$,
with the norm defined by interpolation. Equivalently, the space is equipped with the quotient norm
\begin{align}\label{quotient-norm}
\|v\|_{W^{s,p}(0,T;B)}:=
\inf_{\widetilde v}\|\widetilde v\|_{W^{s,p}({\mathbb R};B)} ,
\end{align}
where the infimum is taken over all possible extensions $\widetilde v$ that extend $v$ from $(0,T)$ to ${\mathbb R}$.
For any $0<s< 1$ and $1 \le p<\infty$, one can define Sobolev--Slobodecki\v{\i} seminorm $|\cdot|_{W^{s,p}(0,T;L^2(\Omega))}$ by
\begin{equation}\label{eqn:SS-seminorm}
   | v  |_{W^{s,p}(0,T;L^2(\Omega))}^p := \int_0^T\int_0^T \frac{\|v(t)-v(\xi)\|_{B}^p}{|t-\xi|^{1+ps}} \,\d t\d\xi ,
\end{equation}
and the full norm $\|\cdot\|_{W^{s,p}(0,T;B)}$ by
\begin{equation*}
\|v\|_{W^{s,p}(0,T;B)}^p = \|v\|_{L^p(0,T;B)}^p+|v|_{W^{s,p}
(0,T;B)}^p .
\end{equation*}

Now we are ready to prove the well-posedness.
\begin{theorem}\label{thm:exists}
For every $u_0\in H_0^1(\Omega)$ there exists a unique weak solution
$$u \in W^{\alpha,p}(0,T;L^2\II) \cap L^p(0,T;H^2\II \cap H_0^1\II),\qquad \text{for any} ~~p\in[2,2/\alpha).$$
Moreover, if $u_0\in H^2\II \cap H_0^1\II$, the solution
$$u\in L^\infty((0,T)\times\Omega).$$ 
\end{theorem}

\begin{proof}
 {\bf Step 1.} Following the routine of the Galerkin method, let
$\{\lambda_j\}_{j=1}^\infty$ and $\{\phi_j\}_{j=1}^\infty$ be respectively the eigenvalues and
the $L^2(\Omega)$-orthonormal eigenfunctions of the negative Laplace operator $-\Delta$ on the domain
$\Omega$ with the homogeneous boundary condition.
For every $N\in \mathbb{N}$, by setting $X_N = \text{span} \{ \phi_j \}_{j=1}^N$,
we consider the finite dimensional problem: find $u^N \in X_N$ such that
\begin{equation}\label{eqn:uN}
\big(\Dal u_N, v \big) - \kappa^2\big(\nabla u_N, \nabla v\big) = \big( f(u^N), v \big)\quad \forall~v\in X_N \quad \text{and}\quad u_N(0) = P_N u_0,
\end{equation}
where $P_N$ is a $L^2$-projection from $L^2\II$ onto $X_N$ by
\begin{equation}\label{eqn:P_N}
\big( u, v \big) =  \big(  P_N u, v \big)\quad \forall~v\in X_N.
\end{equation}
The existence and uniqueness of a local solution to the finite dimensional problem \eqref{eqn:uN} can be proved by the Banach fixed point theorem,
by noting that $f(s)$ is smooth and hence locally Lipshitz continuous \cite[Section 42]{KilbasSrivastavaTrujillo:2006}. Then by the energy argument, we let $v = u_N$ in \eqref{eqn:uN} and using the fact that
\begin{equation*}
\frac12\Dal  \|  u_N(t) \|_{L^2\II}^2 \le \big(\Dal u_N(t), u_N(t)\big).
\end{equation*}
Then we obtain the following estimate
\begin{equation*}
 \Dal \|  u_N(t) \|_{L^2\II}^2 + 2 \kappa^2 \|  \nabla u_N(t)  \|_{L^2\II}^2 +   2\|  u_N (t)\|_{L^4\II}^4 \le  2\| u_N (t)\|_{L^2\II}^2.
\end{equation*}
which together with the fractional Gr\"onwall's inequality in Lemma \ref{Frac-Gronwall} yield  that
$$\|  u_N(t) \|_{L^2\II}  \le c_T \| u_N (0)\|_{L^2\II}  \le c_T \|  u_0  \|_{L^2\II}  ,$$
with a constant $c_T$ independent of $N$. That means $u_N\in C([0,T];L^2\II)$, and hence the solution is a global solution.\vskip5pt

{\bf Step 2.} Now we shall prove that $u_N$ approaches the weak solution of the time-fractional Allen-Cahn equation \eqref{eqn:phase-field}
as $N\rightarrow\infty$.
Now we repeat the energy argument by taking $v=-\Delta u_N$ in \eqref{eqn:uN},
and use the fact that
\begin{equation*}
 -\left(u^N - (u^N)^3,\Delta u_N \right) = \|\nabla u^N\|_{L^2\II}^2 - 3\|  u^N |\nabla u^N| \|^2_{L^2\II}.
 \end{equation*}
Then we arrive at the estimate
 \begin{equation*}
\Dal \|  \nabla u_N (t)\|_{L^2\II}^2 + 2\kappa^2 \|  \Delta u_N  (t)\|^2 +  6\|  u^N |\nabla u^N|(t) \|_{L^2\II}^2 \le 2\| \nabla u_N (t)\|_{L^2\II}^2.
 \end{equation*}
By Lemma \ref{Frac-Gronwall}, it holds that
$$ \|\nabla u_N\|_{L^\infty (0,T; L^2\II)}\le c_T \|  \nabla u_N \|_{L^2\II} \le c_T \|  \nabla u_0  \|_{L^2\II} ,$$
with a constant $c_T$ independent of $N$. As a result, the Sobolev embedding inequality leads to
that $u_N\in L^\infty (0,T; L^6\II)$ for $\Omega\subset \mathbb{R}^d$ with $d \le 3$, and hence
 $$\|f(u_N)\|_{L^\infty(0,T; L^2\II)}\le c.$$ Now under the assumption that $u_0\in H_0^1(\Omega)$, we apply
and the maximal $L^p$ regularity \cite{JinLiZhou:max-reg}, we may obtain  %
   \begin{equation*}
 \|  \partial_t^\alpha u_N  \|_{L^p(0,T;L^2\II)} + \| \Delta u_N \|_{L^p(0,T;L^2\II)} \le c_{p,\kappa}, \qquad \text{for all}~~ p\in[2,2/\alpha).
 \end{equation*}
where the constant $c_{p,\kappa}$ is independent of $N$. This further implies
$$u_N \in W^{\alpha,p}(0,T;L^2\II) \cap  L^p(0,T; H^2 \II\cap H_0^1\II),$$
which is compactly embedded in $C([0,T];L^2\II)$ for $p \in (1/\alpha,2/\alpha)$.
Therefore, there exist function $u$ and a subsequence, which we denote $\{u_N\}$
again, such that
 $$ u_{N} \longrightarrow u \quad\text{weak--* in}\quad L^\infty(0,T;H_0^1\II),  $$
$$\partial_t^\alpha u_{N} \longrightarrow \partial_t^\alpha  u \quad\text{weakly in}\quad L^p(0,T;L^2\II),$$
$$  u_{N} \longrightarrow  u \quad\text{weakly in}\quad L^p(0,T;H^2\II \cap H_0^1\II),$$
 $$ u_{N} \longrightarrow u \quad\text{strongly in }\quad C([0,T];L^2\II),  $$
 as $N\rightarrow0$.  Consequently, we can pass to the limit in \eqref{eqn:uN} and
 the function $u$ satisfies
\begin{equation}\label{eqn:var}
( \partial_t^\alpha u, v ) + \kappa^2 (\nabla u, \nabla v) = ( f(u),v ),\qquad \text{for all} ~~v\in H_0^1\II.
 \end{equation}
  Moreover, by the strong convergence of $u_N$ in $C ([0, T ]; L^2\II)$, we know that $u_N (0)$ converges to $u_0$ in $L^2(\Omega)$ and so $u(0) = u_0$.
Therefore, the fractional Allen-Cahn equation \eqref{eqn:phase-field}  uniquely admits a weak solution in  $W^{\alpha,p}(0,T;L^2\II) \cap L^p(0,T; H^2(\Omega)\cap H_0^1\II)$
with $p\in[2,2/\alpha)$.\vskip5pt

{\bf Step 3.} Finally, we shall prove the $L^\infty((0,T)\times\Omega)$ bound under the assumption on initial data.
The preceding argument implies that $f(u) \in L^\infty(0,T; L^2\II)$. Here we assume that $u_0 \in H^2(\Omega)\cap H_0^1\II$,
and apply the maximal $L^p$ regularity again to obtain that
  \begin{equation*}
 \|  \partial_t^\alpha u  \|_{L^p(0,T;L^2\II)} + \| \Delta u \|_{L^p(0,T;L^2\II)} \le c, \qquad \text{for all}~~ p\in(1,\infty),
 \end{equation*}
which implies that
\begin{equation*}
  u  \in W^{\alpha,p}(0,T;L^2\II) \cap  L^p(0,T;H^2\II\cap H_0^1\II).
\end{equation*}
Therefore, by means of the real interpolation with sufficient large $p$, we obtain that
\begin{equation*}
  u  \in L^\infty((0,T)\times\Omega) .
\end{equation*}
This completes the proof the theorem.
\end{proof}

The next theorem gives the uniqueness of the solution.
\begin{theorem}\label{thm:unqiue}
The weak solution of the fractional-in-time Allen-Cahn equation \eqref{eqn:phase-field} is unique.
\end{theorem}
\begin{proof}
Assume that $u_1$ and $u_2$ are two weak solution of \eqref{eqn:phase-field}. Then $w=u_1-u_2$ satisfies
$$  (\partial_t^\alpha w(t) ,v(t) ) + \kappa^2(\nabla w(t) , v(t))  = ( f(u_1) - f(u_2) , v(t)) ,\quad \text{for all} ~~v\in  H_0^1(\Omega)   $$
with $w(0) = 0$.
By taking $v = w(t)$ and using the facts that
$$( f(u_1) - f(u_2)) (u_1  - u_2) = (u_1 - u_2)^2(1-u_1^2-u_1u_2 - u_2^2) \le  (u_1 - u_2)^2$$
and
$$ (\partial_t^\alpha w, w) \ge \frac12 \partial_t^\alpha \|w(t)\|^2 $$
we obtained that
$$   \partial_t^\alpha \|w(t)\|_{L^2\II}^2  \le  2 \|w(t) \|_{L^2\II}^2 $$
with $w(0)=0$. Applying the fractional Gr\"onwall's inequality in Lemma \ref{Frac-Gronwall},
it follows directly that $w\equiv0$, i.e., $u_1=u_2$.
\end{proof}

With the help of Theorem \ref{thm:exists}, we know that $u\in L^\infty((0,T)\times\Omega)$ and hence $f(u)$ is Globally Lipschitz continuous.
Then the  regularity theory follows from the argument in \cite[Theorem 3.1]{JinLiZhou:nonlinear}.
To this end, we reformulate the fractional Allen-Cahn equation by
\begin{equation}\label{eqn:reform}
\partial_t^\alpha u + A u = u - u^3 \quad \text{for}~(x,t)\in \Omega\times(0,T] \quad \text{and} \quad u(0)=u_0,
\end{equation}
where $A=-\kappa^2\Delta$ with homogeneous Dirichlet boundary condition.
Let $\|\cdot\|_{L^2(\Omega)\to L^2(\Omega)}$ be the operator norm on the space $L^2(\Omega)$. Then the operator $A$ satisfies the following resolvent estimate
\begin{equation*}
  \| (z +A)^{-1} \|_{L^2(\Omega)\to L^2(\Omega)}\le c_{\phi,\kappa} |z|^{-1},  \quad \forall z \in \Sigma_{\phi},
  \,\,\,\forall\,\phi\in(0,\pi) ,
\end{equation*}
where for $\phi\in(0,\pi)$, $\Sigma_{\phi}:=\{z\in{\mathbb C}\backslash\{0\}: |{\rm arg}(z)|<\phi\}$.
This further implies
\begin{equation}\label{eqn:resol}
\begin{aligned}
 & \| (z^{\alpha}+A)^{-1} \|_{L^2(\Omega)\to L^2(\Omega)}
\le c_{\phi,\alpha,\kappa} |z|^{-\alpha},
 &&\forall z \in \Sigma_{\phi} ,
 \,\,\,\forall\,\phi\in(0,\pi) ,\\
& \| A(z^{\alpha}+A)^{-1} \|_{L^2(\Omega)\to L^2(\Omega)} \le c_{\phi,\alpha,\kappa} ,
&&\forall z \in \Sigma_{\phi} ,
 \,\,\,\forall\,\phi\in(0,\pi) .
 \end{aligned}
\end{equation}

Now we consider the solution representation of the following linear equation
\begin{equation}\label{eqn:linear}
    \partial_t^\alpha u(t)   + A u(t)  =  g(t)  ,
\end{equation}
with the homogeneous Dirichlet boundary condition and $u(0)=u_0$. By means of Laplace transform, denoted by $~\widehat{}~$, we obtain
\begin{equation*}
  z^\alpha \widehat{u}(z) + A\widehat w(z) = z^{\alpha-1}u_0+\widehat {g}(z),
\end{equation*}
which together with \eqref{eqn:resol} implies $\widehat{u}(z)=(z^\alpha-A)^{-1}(z^{\alpha-1}u_0+\widehat{g}(z))$. By
inverse Laplace transform and convolution rule, the solution
$u(t)$ to \eqref{eqn:linear} is given by
\begin{align}\label{re-form-nonlinear}
u(t)=F(t) u_0+\int_0^t E(t-s)g(s) d s ,
\end{align}
where the operators $F(t):X \to X $ and $E(t):X \to X $ are defined by 
\begin{equation}\label{eqn:EF}
F(t):=\frac{1}{2\pi {\rm i}}\int_{\Gamma_{\theta,\delta }}e^{zt} z^{\alpha-1} (z^\alpha+A )^{-1}\, \d z
\quad\mbox{and}\quad E(t):=\frac{1}{2\pi {\rm i}}\int_{\Gamma_{\theta,\delta}}e^{zt}  (z^\alpha+A )^{-1}\, \d z ,
\end{equation}
respectively. Clearly, we have $E(t)=F^\prime(t)$. The contour $\Gamma_{\theta,\delta}$ is defined by
\begin{equation}\label{contour-Gamma}
  \Gamma_{\theta,\delta}=\left\{z\in \mathbb{C}: |z|=\delta , |\arg z|\le \theta\right\}\cup
  \{z\in \mathbb{C}: z=\rho e^{\pm {\rm i}\theta}, \rho\ge \delta \},
\end{equation}
oriented with an increasing imaginary part, where $\theta\in(\pi/2,\pi)$ is fixed. \bigskip

The following lemma gives the smoothing properties of $F(t)$ and $E(t)$, which are important in the regularity estimate.
$\rm(i)$, $\rm(ii)$ and $\rm(iii)$ has been proved in \cite[Lemma 2.2]{JinLiZhou:var} and  $\rm(iv)$ was given in \cite[Lemma 3.2]{JinLiZhou:nonlinear}.
\begin{Lemma}\label{lem:smoothing}
For the operators $F$ and $E$ defined in \eqref{re-form-nonlinear}, the
following properties hold.
\begin{itemize}
\item[$\rm(i)$] $ \|F(t) v \|_{L^2\II }
+  t^\alpha \|A  F(t) v\|_{L^2\II }\le c \| v \|_{L^2\II}
,\quad\forall\,t\in(0,T]$ ,
\item[$\rm(ii)$]  $t^{1-\alpha} \|  F'(t)v\|_{L^2\II } + t \| AF'(t)v\|_{L^2\II } +t^{1-\beta\alpha}\|A^{-\beta}F'(t)v\|_{L^2\II}  \le c \| v\|_{L^2\II},\quad\forall\,t\in(0,T]$,
\item[$\rm(iii)$] $t^{1-\alpha}\|E(t)v\|_{L^2\II }+t^{2-\alpha}\|E'(t)v\|_{L^2\II }+t\|A  E(t)v\|_{L^2\II }\le c\| v \|_{L^2\II} , \quad\forall\,t\in(0,T]$;
 \item[$\rm(iv)$]$F(t)$ and $E(t):L^2\II \rightarrow H^2\II\cap H_0^1\II \,\,\,\mbox{is continuous with respect to $t\in(0,T]$}.$
\end{itemize}
\end{Lemma}

\begin{theorem}\label{THM:Reg}
Let $u_0\in H^2\II\cap H_0^1\II$. 
Then the time-fractional Allen-Cahn equation \eqref{eqn:phase-field} has a unique solution $u$ such that for $s\in[0,1)$
\begin{align}
&u\in C^\alpha([0,T];L^2(\Omega)) \cap
C([0,T]; H^2(\Omega)\cap H_0^1\II) ,\quad
\partial_t^\alpha u\in C([0,T];L^2(\Omega)) , \label{reg-PDE}\\
&\Delta u \in  C((0,T];H^{2s}\II ) ,\quad
  \|\Delta^{1+s} u (t)\|_{L^2(\Omega)} \le ct^{-s\alpha} \quad \mbox{for}\,\,\, t\in(0,T] . \label{reg-PDE2}\\
&\partial_tu(t)\in C((0,T];{H^{2s}\II  }) \quad
\mbox{and}\quad
 \| \Delta^s u_t (t)\|_{L^2\II}\le ct^{\alpha(1-s)-1}
\quad \mbox{for}\,\,\, t\in(0,T] .
 \label{reg-PDE3}
\end{align}
The constant $c$ above depends on $\|Au_0\|_{L^2\II}$ and $T$.
\end{theorem}
\begin{proof}
The regularity estimate \eqref{reg-PDE} with $p=2$ has already been proved in \cite[Theorem 3.1]{JinLiZhou:nonlinear}. The result in case that $p\in(2,\infty)$
follows from the same argument and the resolvent estimate \eqref{eqn:resol}.
To show other estimates, we shall apply the representation \eqref{re-form-nonlinear} and smoothing properties in Lemma \ref{lem:smoothing}.
Specifically,
\begin{align*}
A^{1+s} u(t) =  A^{s}F(t) (\Delta u_0)+ \int_0^tA^s E(t-y) [A(u - u^3)(y)] \d y .
\end{align*}
Then we apply Lemma \ref{lem:smoothing} and obtain that
\begin{align*}
 \| A^{1+s} u(t) \|_{L^2\II} \le  c t^{-s\alpha} \| A u_0 \|_{L^2\II}+ \int_0^t y^{(1-s)\alpha-1} \|A (u - u^3)(y)\|_{L^2\II} \d y .
\end{align*}
Now applying the fact that $ \Delta u^3  = 3u |\nabla u|^2 + 3u^2 \Delta u$, $u \in C([0,T];H^2\II\cap H_0^1\II)$ from \eqref{reg-PDE}
and $u \in L^\infty((0,T)\times\Omega)$ by Theorem \ref{thm:exists} we have
\begin{align*}
\|A ((2u - u^3)) \|_{L^2\II} \le c\| \Delta u \|_{L^2\II} + c \|u |\nabla u|^2\|_{L^2\II} + c \| u^2 \Delta u \|_{L^2\II} \le c_T ,
\end{align*}
with a constant $c_T$ independent of $t$. Therefore, we obtain that
\begin{align*}
 \| A^{1+s} u(t) \|_{L^2\II} 
 &\le   c t^{-s\alpha}  .
\end{align*}
Now we turn to \eqref{reg-PDE3}. The case that $s=0$ has be confirmed in \cite[Theorem 3.1]{JinLiZhou:nonlinear}. In general,
\begin{align*}
A^{s} u_t(t) & =  A^{s-1}F'(t) (A u_0)+ \frac{d}{dt}\int_0^t A^s E(y) [(u - u^3)(t-y)] \d y \\
&= A^{s-1}F'(t) (A u_0) + A^s E(t) [(2u - u^3)(0)] + \int_0^t A^s E(y)  [(u_t - 3u^2 u_t)(t-y)] \d y
\end{align*}
The first term could be bounded using Lemma \ref{lem:smoothing} (ii)
\begin{align}\label{eqn:e1}
  \| A^{s-1}F'(t) (A u_0) \|_{L^2\II} \le  ct^{(1-s)\alpha-1}  \| \Delta u_0 \|_{L^2\II}.
\end{align}
For the second term, we use  Lemma \ref{lem:smoothing} (iii) to have
\begin{align}\label{eqn:e2}
 \| A^s E(t) [(2u - u^3)(0)] \|_{L^2\II} \le c t^{(1-s)\alpha-1}  \| 2u_0 - u_0^3 \|_{L^2\II} \le c t^{(1-s)\alpha-1} \| \Delta u_0\|_{L^2\II}.
\end{align}
Similarly, the third term follows analogously
\begin{equation*}
\begin{split}
 \int_0^t  \| A^s E(y) (2u_t - 3u^2 u_t)(t-y)\|_{L^2\II} \d y &\le  c \int_0^t y^{\alpha-1}\| A^s u_t(t-y)\|_{L^2\II} \d y.
\end{split}
\end{equation*}
This together with \eqref{eqn:e1} and \eqref{eqn:e2}, and the Gr\"onwall's inequality lead to the desired result.

Finally, the continuity follows directly from the continuity of the solution operators $E(t)$ and $F(t)$, i.e., Lemma \ref{lem:smoothing} (iv).
\end{proof}

Next, we shall prove the maximum principle
in sense that
if $|u_0|\le 1$, then the solution $u(t)$ is also bounded by $1$ in
$L^\infty(\Omega)$. Note that the maximum principle was showed in \cite{Tang:2018}
under certain a priori strong regularity assumptions
on the solution $u$, which have not been rigorously proven yet.

\begin{Lemma}\label{lem:Ext}
{Assume that $f\in C[0,T]\cap C^1(0,T]$
and  $f$ attains its minimum (maximum) at  $t_0\in(0,T]$. Then there holds that
$$\Dal f(t_0)\le (\ge)~0.$$}
\end{Lemma}

\begin{theorem}\label{thm:max}
Let $u_0\in H^2\II\cap H_0^1\II$.
Further, we assume that $|u_0(x)|\le 1$, then the solution of the time-fractional Allen Cahn equation \eqref{eqn:phase-field} with homogeneous boundary conditions satisfies
the maximum principle that
$$|u(x,t)|\le 1\qquad \text{for all} \quad (x,t)\in\Omega\times(0,T].$$
\end{theorem}
\begin{proof}
Assume that the minimum of $u$
is smaller than $-1$ and achieved at $(x_0,t_0)\in \Omega\times(0,T]$.
Then by the regularity \eqref{reg-PDE} and \eqref{reg-PDE3} in Theorem \ref{THM:Reg}, and Lemma \ref{lem:Ext}, we have that
$$\Dal u(x_0,t_0)\le 0.$$
Meanwhile, via applying Theorem \ref{THM:Reg}, we have that $\Delta u$ is continuous in $\overline \Omega$ and hence
 $$\Delta u(x_0,t_0) \ge 0.$$
Therefore, we arrive at the result
\begin{equation*}
  0 \ge \Dal u(x_0,t_0) - \Delta u(x_0,t_0) = u(x_0,t_0) - u(x_0,t_0)^3 >0,
\end{equation*}
which leads to a contradiction.
The upper bound of the solution can be proved using the same way.
\end{proof}

\begin{remark}\label{rem:energy-stable}
We are also interested in the energy bound of the solution.
Suppose that the solution satisfies
$u\in W^{1,\frac{2}{2-\alpha}}(0,T; L^2(\Omega))\cap L^2(0,T; H^2\II\cap H_0^1\II)$,
then the stability of energy can be directly observed by taking $v = u_t$ in \eqref{eqn:var} and integrate over $(0,T)$,
\begin{equation*}
\int_0^T (\partial_t^\alpha u (t), u_t(t)) +  (\nabla u (t),\nabla u_t(t)) \,\d t = \int_0^T (f (u(t)), u_t(t))\,\d t
\end{equation*}
Using the regularity assumption and applying the fact that (see \cite[Corollary 2.1.]{Tang:2018} or \cite[Lemma 3.1]{Mustapha:2014})
\begin{equation*}
\int_0^T (\partial_t^\alpha u (t), u_t(t)) \,dt \ge 0,
\end{equation*}
we obtain that
\begin{equation*}
\frac{1}{2}\int_0^T\frac{d}{dt}  \int_\Omega  |\nabla u (t)|^2 \,\d x\,\d t + \int_0^T \frac{d}{dt} \int_\Omega F(u(t)) \,\d x \,\d t \le 0.
\end{equation*}
and hence we derive the energy bound that
\begin{equation*}
E(u(T)) \le E(u_0).
\end{equation*}
This property has been proved in \cite{Tang:2018} without confirmation of the regularity assumption.
By Theorem \ref{THM:Reg}, one can easily verify those assumption, provided that $u_0\in H^2(\Omega)\cap H_0^1(\Omega)$.
\end{remark}

\section{Numerical schemes for the time-fractional Allen-Cahn equation}\label{sec:scheme}
In this section, we shall propose time stepping schemes and study some quantitative properties, such as discrete maximum principle and energy dissipation,
which are important for phase field models.
To this end, we descretize the Caputo fractional derivative by using
 the convolution quadrature \cite{Lubich:1988} generated by backward Euler method (BE-CQ),
which is commonly known as the Gr\"unwald-Letnikov scheme in the literature \cite{OldhamSpanier:1974}.
We divide the interval $[0,T]$ into a uniform grid
with a time step size $\tau = T/N$, $N\in\mathbb{N}$, so that $0=t_0<t_1<\ldots<t_N=T$, and $t_n=n
\tau$, $n=0,\ldots,N$. Then the Caputo fractional derivative is approximated by
 \begin{equation}\label{eqn:BE-CQ}
 \bar \partial_\tau^\alpha u_n  := \tau^{-\alpha}\sum_{j=1}^n \omega_{n-j} (u_j-u_0),\quad \text{where}~~
\omega_j= (-1)^j\frac{\alpha(\alpha-1)\ldots(\alpha-j+1)}{j!}.
\end{equation}
\subsection{Convex splitting scheme and the energy stability.}
The first method is the convex splitting scheme (CS).
For given $u_0,u_1,\ldots,u_{n-1}$, we find $u_n$ by solving a nonlinear elliptic problem
\begin{equation}\label{eqn:cs}
\bar\partial_\tau^\alpha u_n - \kappa^2\Delta u_n=u_{n-1}-(u_{n})^3, \qquad \text{for}~~ 1 \le n\le N,
\end{equation}
with the homogeneous Dirichlet boundary condition.
This simple time stepping method is a nonlinear scheme inspired by the standard convex splitting method for general gradient flows.
which handles the convex part $u^3$ implicitly and replaces the concave part $u_n$ with $u_{n-1}$.
Meanwhile, the CS scheme \eqref{eqn:cs} can be written as the nonlinear elliptic problem
\begin{equation}\label{eqn:cs-1}
\left(I-\tau^\alpha \kappa^2\Delta\right)u_n+\tau^\alpha(u_n)^3=  u_{n,\alpha} + \tau^{\alpha} u_{n-1}.
\end{equation}
where $u_{n,\alpha}$ denotes the fractional extrapolation \eqref{eqn:frac-interp2}.
Since left part is monotone with respect to $u_n$. By the implicit function theorem, there exists a unique solution and hence the CS scheme \eqref{eqn:cs} is uniquely solvable.

The following theorem states that the CS scheme satisfies the discrete maximum principle.
This is an important and well-known property of the standard CS scheme of the classical Allen-Cahn equation.
More investigations on maximum principle preserving schemes for integer order Allen--Cahn equations can be found in \cite{STY1,TY1}.
\begin{theorem}\label{thm:dmp-cs}
The CS scheme \eqref{eqn:cs} satisfies the discrete maximum principle unconditionally, i.e.,
\begin{equation*}
 \|u_0\|_{L^\infty(\Omega)}\le 1 \Longrightarrow \|u_n\|_{L^\infty(\Omega)}\le 1 \quad \text{for all}\quad n\ge 1.
\end{equation*}
\end{theorem}
\begin{proof}
We prove the discrete maximum principle  by induction. To this end, we assume that $\|u_k\|_{L^\infty(\Omega)}\le 1$ for all $k<n$.
Note that the weights $\{ \omega_j \}$ satisfies the property that
\begin{equation*}
		(\text{i})\qquad	\omega_0>0 \ \ \text{and}\ \
		\omega_j<0\ \text{for}\  j\ge1,
	\quad\text{and}\qquad (\text{ii})\qquad
	\displaystyle \sum_{j=0}^n \omega_j>0, \text{for}\  n\ge1.
	\qquad\qquad
\end{equation*}
Thus $u_{n,\alpha}$ is a convex combination of $u_0, u_1,\ldots,u_{n-1}$, and hence $\|u_{n,\alpha}\|_{L^\infty(\Omega)}\le 1$. Therefore, we have
 $\|u_{n,\alpha} + \tau^\alpha u_{n-1}\|_{L^\infty(\Omega)}\le 1+\tau^\alpha$.
Besides, the operator $I-\tau^\alpha \kappa^2\Delta$ is  positive and $(u^n)^3$ has the same sign of $u^n$.
Then it follows immediately that $\|u^n\|_\infty\le 1$   by the monotonicity and \eqref{eqn:cs-1}.
\end{proof}

Besides the discrete maximum principle, we shall prove the energy stability of the time-stepping scheme \eqref{eqn:cs}.
To this end, we shall use the following lemma on the backward Euler convolution quadrature (BE-CQ).
A similar result has been proved in \cite[Lemma 3.1]{Tang:2018} for the time stepping method using piecewise linear polynomial interpolation,
known as L1 approximation in the literature \cite{LinXu:2007, SunWu:2006}.
However, the convolution quadrature does not satisfy the algebraic property like the L1 scheme. 
So here we apply a different approach, say discrete time Fourier transform.

\begin{Lemma}\label{lem:cq-int}
The BE-CQ satisfies the following positivity property
\begin{equation*}
  \sum_{j=1}^N (\bar \partial_\tau^\alpha u_n , \bar \partial_\tau^1 u_n) \ge 0.
\end{equation*}
\end{Lemma}

\begin{proof}
First, we assume that $u_0=0$. Then by setting $y_n= \partial_\tau^{1} u_n$ for $n\ge1$ and $y_0=0$,
the definition of convolution quadrature yields that
\begin{equation*}
  \sum_{j=1}^N (\bar \partial_\tau^\alpha u_n , \bar \partial_\tau^1 u_n) =  \sum_{j=1}^N (\bar \partial_\tau^{\alpha-1} y_n , y_n).
\end{equation*}
where
\begin{equation}\label{eqn:alpha-1}
   \bar \partial_\tau^{\alpha-1} y_n = \tau^{1-\alpha}\sum_{j=0}^n \omega_{n-j}^{(\alpha-1)} y_j\quad \text{with}~~
\omega_j= (-1)^j\frac{(\alpha-1)\ldots(\alpha-j )}{j!}.
\end{equation}
To show this nonnegativity property, we shall extend
$\{ y_n \}_{n=0}^N$ to $\{ y_n \}_{n=-\infty}^{n=\infty}$ and $\{ \omega_n \}_{n=0}^{n=\infty}$ to
$\{ \omega_n \}_{n=-\infty}^{n=\infty}$ by zero extension. Then
 the discrete fractional derivative can be written as
\begin{equation}\label{eqn:alpha-2}
\bar \partial_\tau^{\alpha-1} u_n = \tau^{1-\alpha}\sum_{j=-\infty}^\infty \omega_{n-j}^{(\alpha-1)} y_j
\end{equation}
From now on, we denote the discrete time Fourier transform $ \widetilde {[u_n]}(\zeta)$ by
\begin{equation}\label{eqn:DTFT}
  \widetilde {[u_n]}(\zeta) = \sum_{n=-\infty}^\infty u_n e^{-\mathrm{i} n\zeta}.
\end{equation}
Then by the Parseval's theorem we have
\begin{equation*}
 \sum_{j=0}^N (\bar \partial_\tau^{\alpha-1} y_n , y_n) = \frac{1}{2\pi}\int_{-\pi}^\pi  (\widetilde{[\bar \partial_\tau^{\alpha-1} y_n]}(\zeta) , \widetilde{[y_n]}(\zeta)^*) \,\d\zeta
\end{equation*}
By the property of discrete time Fourier transform \eqref{eqn:DTFT} and the discrete convolution \eqref{eqn:alpha-1}, we have
\begin{equation*}
\begin{split}
 \sum_{j=0}^N (\bar \partial_\tau^{\alpha-1} y_n , y_n) &= \frac{\tau^{1-\alpha}}{2\pi}\int_{-\pi}^\pi \widetilde{[\omega_n^{(\alpha-1)}]}~ \Big|\widetilde{[y_n]}(\zeta)\Big|^2 \,\d\zeta\\
 & = \frac{\tau^{1-\alpha}}{2\pi}\int_{-\pi}^\pi \Big( 1-e^{-\mathrm{i}\zeta} \Big)^{\alpha-1} \Big|\widetilde{[y_n]}(\zeta)\Big|^2 \,\d\zeta\\
 & = \frac{\tau^{1-\alpha}}{2\pi}\int_{0}^\pi 2\Big[\mathbb{Re}\Big( 1-e^{-\mathrm{i}\zeta} \Big)^{\alpha-1}\Big] \Big|\widetilde{[y_n]}(\zeta)\Big|^2 \,\d\zeta \ge 0.\\
 \end{split}
\end{equation*}
In case that $u_0\neq0$, we let $v_n=u_n-u_0$ and note that
\begin{equation*}
  \sum_{j=1}^N (\bar \partial_\tau^\alpha u_n , \bar \partial_\tau^1 u_n) =  \sum_{j=1}^N (\bar \partial_\tau^\alpha v_n , \bar \partial_\tau^1 v_n) .
\end{equation*}
The repeating the argument for $v_n$ leads to the desired result. That completes the proof of the lemma.
\end{proof}

Then the next theorem states that the CS scheme \eqref{eqn:cs} is energy stable.
\begin{theorem}\label{thm:wes-cs}
Suppose that $u_0\in H_0^1\II$. Then the CS scheme \eqref{eqn:cs} satisfies the energy stability
$$E(u_n)\le E(u_{0}),\qquad \text{for all}\quad n\ge 1.$$
\end{theorem}

\begin{proof}
Taking $L^2$-inner product of \eqref{eqn:cs} with $-(u_n-u_{n-1})$ yields
$$-(\bar\partial_\tau^\alpha u_n, u_n - u_{n-1})=\kappa^2(\nabla u_n,\nabla u_n-\nabla u_{n-1})+\left((u_n)^3-u_{n-1},u_n-u_{n-1}\right).$$
Here we note that
$$\kappa^2(\nabla u_n,\nabla u_n-\nabla u_{n-1}) = \frac{\kappa^2}2\| \nabla u_n \|_{L^2\II}^2 -   \frac{\kappa^2}2\| \nabla u_{n-1} \|_{L^2\II}^2
+ \frac{\kappa^2}{2}\| \nabla(u_n-u_{n-1}) \|_{L^2\II}^2,$$
Meanwhile, the fundamental inequality
\begin{equation}\label{eqn:u3}
(a^3-b,a-b)\ge \frac{1}{4}\left(a^2-1\right)^2-\frac{1}{4}\left(b^2-1\right)^2
\end{equation}
yields that
\begin{equation}\label{eqn:u33}
\left((u_n)^3-u_{n-1},u_n-u_{n-1}\right) \ge \frac{1}{4}\|(u_n)^2-1\|_{L^2\II}^2-\frac{1}{4}\|(u_{n-1})^2-1\|_{L^2\II}^2.
\end{equation}
Therefore, we arrive at the estimate
\begin{equation*}
\begin{split}
-(\bar\partial_\tau^\alpha u_n, u_n - u_{n-1}) \ge E(u_n) - E(u_{n-1})
\end{split}
\end{equation*}
Summing up both sides for $n = 1,\ldots,N$, we obtain that
\begin{equation*}
\begin{split}
-\tau \sum_{n=1}^N(\bar\partial_\tau^\alpha u_n, \bar\partial_\tau^1 u_n) \ge E(u_N) - E(u_{0}).
\end{split}
\end{equation*}
Using Lemma \ref{lem:cq-int}, we have that
\begin{equation*}
E(u_N) -  E(u_{0}) \le 0
\end{equation*}
which completes the proof.
\end{proof}

\begin{remark}
However, little is known about the energy dissipation law, which is important to the gradient flow models. For the convex splitting scheme of the conventional Allen-Cahn eqaution,
it is well-known that the energy at each time level is decreasing, i.e., $E(u_n) \le E(u_{n-1})$ for all $n\ge 1$, which might not be true in the fractional case.
Therefore, to develop a time stepping scheme which satisfies some energy dissipation principles is  a very interesting
and important task, which will be explored in the next part.
\end{remark}

\subsection{Fractional weighted schemes and energy dissipation law.}
To derive some novel time stepping schemes, we shall introduce the following backward fractional interpolation
\begin{equation}\label{eqn:frac-interp}
I_\alpha u(t_n)=  - \sum_{j=1}^{n-1} \omega_{n-j} \big(u(t_{n-j})-u(0)\big) + \omega_0 u(0),
\end{equation}
which is independent of $u(t_n)$.
Then it is easy to see that
\begin{equation}\label{eqn:frac-interp2}
u(t_n) - I_\alpha u(t_n) =  \sum_{j=1}^{n} \omega_{n-j} (u(t_j)-u(0)) =\tau^\alpha \bar \partial_\tau^\alpha u(t_n),
\end{equation}
which means the approximation error is of order $O(\tau^\alpha)$.

Then we propose the so-called weighted convex splitting (WCS) scheme, reading that for given $u_0, u_1,\ldots,u_{n-1}$, we shall look for the function $u_n$ satisfying
\begin{equation}\label{eqn:wcs}
\bar\partial_\tau^\alpha u_n - \kappa^2\Delta u_n = u_{n,\alpha}-(u_n)^3, \qquad \text{for}~~ 1 \le n\le N,
\end{equation}
with the homogeneous Dirichlet boundary condition,
where $u_{n,\alpha}$ denotes the fractional extrapolation \eqref{eqn:frac-interp2}.
In this nonlinear scheme, we handle the convex part $u^3$ implicitly and replaces the concave part $u$ with the fractional extrapolation $u_{n,\alpha}$.
The WCS scheme \eqref{eqn:wcs} can be written as the following nonlinear elliptic problem
\begin{equation}\label{eqn:wcs2}
\left(I-\tau^\alpha \kappa^2\Delta\right)u_n+\tau^\alpha(u_n)^3=\left(1+ \tau^{\alpha} \right)u_{n,\alpha}
\end{equation}
Since the left hand side is monotone with respect to $u_n$, the unique solvability follows directly from the implicit function theorem.

The next time stepping scheme, called linear weighted stabilized (LWS) scheme, is developed as
\begin{equation}\label{eqn:lws}
\bar\partial_\tau^\alpha u_n - \kappa^2\Delta u_n  = u_{n,\alpha}-(u_{n,\alpha})^3-S\left(u_n-u_{n,\alpha}\right), \qquad \text{for}~~ 1 \le n\le N,
\end{equation}
where $S>0$ is a stabilization constant. Such the scheme is linear because $u_{n,\alpha}$ is independent of $u_n$.
Then at each time level, the scheme \eqref{eqn:lws} requires  solving the elliptic problem
\begin{equation*}
\left((1+\tau^\alpha S)I - \tau^\alpha \kappa^2\Delta\right)u_n =  (1+ (S+1) \tau^\alpha) u_{n,\alpha}- \tau^\alpha (u_{n,\alpha})^3
\end{equation*}
with the homogeneous Dirichlet boundary condition, which is uniquely solvable.

Next, we intend to show that both of schemes \eqref{eqn:wcs} and \eqref{eqn:lws} satisfy the discrete maximum principle.
\begin{theorem}\label{thm:dmp}
The WCS scheme \eqref{eqn:wcs} satisfies the discrete maximum principle unconditionally, i.e.,
\begin{equation*}
 \|u_0\|_{L^\infty(\Omega)}\le 1 \Longrightarrow \|u_n\|_{L^\infty(\Omega)}\le 1 \quad \text{for all}\quad n\ge 1.
\end{equation*}
Meanwhile, the LWS scheme \eqref{eqn:lws} satisfies the discrete maximum principle with the constraint $$S+ \tau^{-\alpha} \ge2,$$
where $S$ is the stabilization constant in \eqref{eqn:lws}.
\end{theorem}
\begin{proof}
We prove the discrete maximum principle  by induction. Recall that
$u_{n,\alpha}$ is a convex combination of $u_0, u_1,\ldots,u_{n-1}$, and hence $\|u_{n,\alpha}\|_{L^\infty(\Omega)}\le 1$.
Again, we rewrite the WCS scheme \eqref{eqn:wcs} into the form \eqref{eqn:wcs2} 
and note that the operator $I-\tau^\alpha \kappa^2\Delta$ is  positive. Meanwhile, $(u_n)^3$ has the same sign of $u^n$.
Hence, it follows immediately that $\|u^n\|_\infty\le 1$ by the monotonicity. \vskip5pt
Now we turn to the LWS scheme \eqref{eqn:lws}, we rewrite the time stepping scheme as
\begin{equation*}
\left((S+\tau^{-\alpha})I-\kappa^2\Delta\right)u_n=\left(1+S+\tau^{-\alpha} \right)u_{n,\alpha}-(u_{n,\alpha})^3.
\end{equation*}
The right side in the form of $h(x)=(1+c)x-x^3$ with the constant $c=S+\tau^{-\alpha}$ and $x\in[-1,1]$.
Under the assumption that $c=S+\tau^{-\alpha}\ge2$,
the function $h(x)$ is monotonically increasing as well as
$$\min_{x\in[-1,1]}h(x)=h(-1)=-c\quad\text{and}\quad \max_{x\in[-1,1]}h(x)=h(1)=c,$$
due to $h'(x)=(1+c)-3x^2 \ge 0$. As a result, we have
$$\left\|\left(1+S+ \tau^{-\alpha} \right) u_{n,\alpha}-(u_{n,\alpha})^3\right\|_{L^\infty\II}\le S+ \tau^{-\alpha},$$
which together with the property that
$$\left\|\left( (S+ \tau^{-\alpha})I-\kappa^2\Delta\right)^{-1}\right\|_{L^\infty\II\rightarrow L^\infty\II} \le \left(S+ \tau^{-\alpha} \right)^{-1}$$
compete the proof of the theorem.
\end{proof}

Finally, we intend to study the energy dissipation property.
\begin{theorem}\label{thm:wes-wcs}
The WCS scheme \eqref{eqn:wcs} satisfies the weighted energy stability unconditionally, i.e.,
$$E(u_n)\le E(u_{n,\alpha}),\qquad \text{for all}\quad n\ge 1.$$
The LWS scheme \eqref{eqn:lws} also satisfies the weighted energy stability with the constraint $S+\tau^{-\alpha}\ge2$.
\end{theorem}
\begin{proof}
Taking $L^2$-inner product of \eqref{eqn:wcs} with $-(u_n-u_{n,\alpha})$ yields
$$-\frac{1}{\Gamma(2-\alpha)\tau^{\alpha}}\left\|u_n-u_{n,\alpha}\right\|_{L^2}^2=\kappa^2(\nabla u_n,\nabla u_n-\nabla u_{n,\alpha})+\left((u_n)^3-u_{n,\alpha},u_n-u_{n,\alpha}\right).$$
By a fundamental equaliy
$$a(a-b)=\frac{1}{2}a^2-\frac{1}{2}b^2+\frac{1}{2}(a-b)^2,$$
and a fundamental inequality
$$(a^3-b,a-b)\ge \frac{1}{4}\left(a^2-1\right)^2-\frac{1}{4}\left(b^2-1\right)^2,$$
we have  
$$
   \begin{aligned}
  & \frac{\kappa^2}{2}\|\nabla u_n\|_{L^2}^2-\frac{\kappa^2}{2}\|\nabla u_{n,\alpha}\|_{L^2}^2+\frac{1}{4}\|(u_n)^2-1\|_{L^2}^2-\frac{1}{4}\|(u_{n,\alpha})^2-1\|_{L^2}^2 \\
   \le&-\frac{1}{\Gamma(2-\alpha)\tau^{\alpha}}\left\|u_n-u_{n,\alpha}\right\|_{L^2}^2-\frac{\kappa^2}{2}\|\nabla u_n-\nabla u_{n,\alpha}\|_{L^2}^2.
   \end{aligned}
  $$
  Thus, for the WCS scheme \eqref{eqn:wcs} the weighted energy stability $E(u_n)\le E(u_{n,\alpha})$ holds unconditionally.

\bigskip For the LWS scheme \eqref{eqn:lws}, we follow the similar technique, we can derive
  $$\kappa^2(\nabla u_n,\nabla u_n-\nabla u_{n,\alpha})+\left((u_{n,\alpha})^3-u_{n,\alpha},u_n-u_{n,\alpha}\right)+\left(S+ \tau^{-\alpha}\right)\left\|u_n-u_{n,\alpha}\right\|_{L^2}^2=0.$$
  We again use the fundamental equality
$$a(a-b)=\frac{1}{2}a^2-\frac{1}{2}b^2+\frac{1}{2}(a-b)^2,$$
and obtain the following inequality
$$(b^3-b)(a-b)+(a-b)^2\ge \frac{1}{4}(a^2-1)^2-\frac{1}{4}(b^2-1)^2, \quad \forall a,b\in[-1,1].$$
As in last theorem we have shown the discrete maximum principle, it follows above inequality that
$$\left((u_{n,\alpha})^3-u_{n,\alpha},u_n-u_{n,\alpha}\right)+\left(S+ \tau^{-\alpha}\right)\left\|u_n-u_{n,\alpha}\right\|_{L^2}^2\ge \frac{1}{4}\|(u_n)^2-1\|_{L^2}^2-\frac{1}{4}\|(u_{n,\alpha})^2-1\|_{L^2}^2,$$
whenever $S+\tau^{-\alpha}\ge 2$. This ends up with
$$
  \frac{\kappa^2}{2}\|\nabla u_n\|_{L^2}^2-\frac{\kappa^2}{2}\|\nabla u_{n,\alpha}\|_{L^2}^2+\frac{1}{4}\|(u_n)^2-1\|_{L^2}^2-\frac{1}{4}\|(u_{n,\alpha})^2-1\|_{L^2}^2 \le 0,
  $$
  which is exactly the desired result.
\end{proof}

\begin{remark}\label{rem:egy-disc}
It is seen that the weighted energy stability can be extended for more general form free energy.
We intend to emphasize that if the free energy is given convex (e.g., the linear problem or time-fractional
Allen-Cahn with $|u_n|\ge \frac{\sqrt3}{3}$), the weighted energy stability indicates the fractional energy dissipation law as
\begin{equation}\label{eqn:egy-disc}
\bar\partial_\tau^\alpha E(u_n)\le 0.
\end{equation}
This is consistent with the energy decay property of classical gradient flow ($\alpha=1$), i.e.,
$\bar\partial_\tau E(u_n)\le 0.$
At the continuous level, it is possible to derive the fractional energy dissipation law for the linear subdiffusion model
\begin{align}\label{eqn:fde-linear}
\partial_t^\alpha u - \kappa^2 \Delta u=0.
\end{align}
By taking $L^2$-inner product of \eqref{eqn:fde-linear} with $-\partial_t^\alpha u$, we derive that
\begin{equation*}
 0\ge - \| \partial_t^\alpha u(t) \|_{L^2\II}^2 = \kappa^2 (\nabla u (t), \Dal \nabla u(t)) \ge \Dal \Big(\frac{1}{2} \|  \nabla u (t)\|_{L^2\II}^2\Big) =  \Dal E(u(t)),
\end{equation*}
which is consistent with the discrete energy dissipation law \eqref{eqn:egy-disc}.
Unfortunately, the free energy associated with Allen--Cahn equation is nonlinear and not always convex with respect to $u$.
So far, we cannot prove the fractional energy dissipation law but many numerical examples have already verified
it in some literatures. Seeking the theoretical proof the  fractional energy dissipation law is a very
interesting and meaningful future work.
\end{remark}

\section{Error analysis of time stepping schemes}\label{sec:error}
In this section, we shall derive the error analysis of the time stepping schemes proposed in Section \ref{sec:scheme}.
This requires some preliminary estimate for linear problem \eqref{eqn:linear} and the implicit BE-CQ scheme
\begin{equation}\label{eqn:BE-CQ-linear}
\bar\partial_\tau^\alpha u_n + A u_n = g(t_n)  \qquad \text{for}~~ 1 \le n\le N,
\end{equation}
with given initial condition $u_0$,  where $A=-\kappa^2 \Delta$ with the homogeneous Dirichlet boundary condition. The following lemma gives the error estimate,
which has been developed in \cite{JinLazarovZhou:SISC2016}.

\begin{Lemma}\label{lem:error-linear}
Let $u_0\in H^2\II\cap H_0^1\II$ and $u(t)$ be the solution of the linear time-fractional evolution equation \eqref{eqn:linear}.
Then $\{u_n\}$, the solutions of implicit BE-CQ scheme \eqref{eqn:BE-CQ-linear}  satisfies
\begin{equation*}
 \| u_n - u(t_n) \|_{L^2\II} \le c  t_n^{\alpha-1}\tau \|A u_0 + g(0)\|_{L^2(\Omega)} +c\tau\int_0^{t_n}(t_n-s)^{\alpha-1}\|g'(s)\|_{L^2(\Omega)}\d s.
\end{equation*}
\end{Lemma}

For $1\le p\leq \infty$ and $X$ being a Banach space, we denote by $\ell^p(X)$ the space of sequences $v^n\in X$, $n=0,1,\dots$,
such that
$\|(v^n)_{n=0}^\infty\|_{\ell^p(X)}<\infty$, where
$$
\|(v^n)_{n=0}^\infty\|_{\ell^p(X)}:=
\left\{
\begin{aligned}
&\bigg(\sum_{n=0}^\infty\tau\|v^n\|_{X}^p\bigg)^{\frac{1}{p}}  &&\mbox{if}\,\,\, 1\le p<\infty,\\
&\sup_{n\ge 0}\|v^n\|_{X} &&\mbox{if}\,\,\, p=\infty .
\end{aligned}\right.
$$
For a finite sequence $v^n\in X$, $n=0,1,\dots,m$, we denote
$\|(v^n)_{n=0}^m\|_{\ell^p(X)}:=\|(v^n)_{n=0}^\infty\|_{\ell^p(X)}$,
by setting $v^n=0$ for $n>m$. The next two lemma shows that the BE-CQ scheme \eqref{eqn:BE-CQ-linear}
satisfies the discrete fractional Gr\"onwall's inequality and discrete maximal $\ell^p$ regularity, whose complete proof is given in
\cite[Theorem 2.2]{JinLiZhou:nonlinear} and \cite[Theorem 5]{JinLiZhou:max-reg}.

\begin{Lemma}\label{lem:disc-Gronwall}
Let $\bar\partial_\tau^\alpha\,$ denote the BE-CQ
given by \eqref{eqn:BE-CQ}.
If $\alpha\in(0,1)$ and $p\in (1/\alpha,\infty)$, and a sequence
$v^n\in X$, $n=0,1,2,\dots$, with $v^0=0$, satisfies
\begin{equation*}
\|(\bar\partial_\tau^\alpha v^n)_{n=1}^m\|_{\ell^p(X)}
\le c_0 \|(v^n)_{n=1}^m\|_{\ell^p(X)}+\sigma ,
\quad\forall\, 0\le m\le N ,
\end{equation*}
for some positive constants $ \kappa$ and $\sigma$, then
there exists a $\tau_0>0$ such that for any $\tau<\tau_0$ there holds
\begin{equation*}
\|(v^n)_{n=1}^N\|_{\ell^\infty(X)}+\|(\bar\partial_\tau^\alpha v^n)_{n=1}^N\|_{\ell^p(X)}
\le c  \sigma ,
\end{equation*}
where the constants $c$ and $\tau_0$ are independent of $\sigma$, $\tau$, $N$, $X$ and $v^n$,
but may depend on $\alpha$, $p$, $c_0$ and $T$.
\end{Lemma}

\begin{Lemma}\label{lem:max-reg}
The BE-CQ scheme \eqref{eqn:BE-CQ-linear} with $u_0 = 0$ has the following maximal $\ell^p$-regularity
\begin{equation*}
  \|(\bar\partial_\tau^\alpha u^n)_{n=1}^m\|_{\ell^p(L^2\II)}+ \|(Au^n)_{n=1}^m\|_{\ell^p(L^2\II)}\leq c\|(f^n)_{n=1}^m\|_{\ell^p(L^2\II)},
\end{equation*}
where the constant $c$ is independent of $N, \tau$.
\end{Lemma}

\subsection{Error analysis of the CS scheme.}
With the help of error estimate in Lemma \ref{lem:error-linear}, the discrete fractional
Gr\"onwall's inequality in Lemma \ref{lem:disc-Gronwall}, the discrete maximal $\ell^p$ regularity in Lemma \ref{lem:max-reg} and the regularity estimate in Section \ref{sec:PDE},
we can derive the error analysis of the CS scheme \eqref{eqn:cs}.
To this end,
%
we use the splitting that
\begin{equation}\label{eqn:error-split}
 u(t_n)- u_n =  (u(t_n) - v_n) + (v_n  - u_n) =: \theta_n + \rho_n .
\end{equation}
where $v_n$ satisfies the time stepping scheme
\begin{align}\label{eqn:vn-1}
 \bar\partial_\tau^\alpha  v_n  + A  v_n = f(u(t_n)),\qquad \text{with}~~v_0 = u_0,
\end{align}

\begin{Lemma}\label{lem:cs-01}
Let $u_0\in H^2\II\cap H_0^1\II$ and $u(t)$ be the solution of the time-fractional Allen-Cahn equation \eqref{eqn:phase-field}.
Then $\theta_n = u(t_n) - v_n$, where $v_n$ is the solutions of \eqref{eqn:vn-1},  satisfies
\begin{equation}\label{error-estimate-vn}
\max_{1\le n\le N}  \|\theta_n\|_{L^2(\Omega)}\le c\tau^\alpha .
\end{equation}
where the constant $c$ may depends on $\alpha, \kappa, T, u_0$ but is independent of $u$ and $\tau$.
\end{Lemma}

\begin{proof}
Using Lemma \ref{lem:error-linear}, we have the following estimate that
\begin{align*}
\begin{aligned}
\|u_h(t_n)-v_h^n\|_{L^2(\Omega)}
\le& ct_n^{\alpha-1}\tau\bigg(\|A u_0 + g(u_0)\|_{L^2(\Omega)}\bigg) +c\tau\int_0^{t_n}(t_n-s)^{\alpha-1}\|f'(u(s))\partial_su(s)\|_{L^2(\Omega)}\d s. \\
\end{aligned}
\end{align*}
Using the fact that $u\in L^\infty((0,T)\times \Omega)$ by Theorem \ref{thm:exists}, we have $\| f'(u(s)) \|_{L^2\II} \le c$, and hence
\begin{equation*}
\|f'(u(s))\partial_su(s)\|_{L^2(\Omega)} \le c \|\partial_su(s)\|_{L^2(\Omega)} \le c s^{\alpha-1},
\end{equation*}
where the last inequality is given by \eqref{reg-PDE3}. As a result, we derive that
\begin{align*}
\begin{aligned}
\|u_h(t_n)-v_h^n\|_{L^2(\Omega)}
\le  ct_n^{\alpha-1}\tau  +c\tau\int_0^{t_n}(t_n-s)^{\alpha-1}s^{\alpha-1}\d s \le c\tau (t_n^{\alpha-1} + t_n^{2\alpha-1}) \le c\tau t_n^{\alpha-1} \le c\tau^\alpha. \\
\end{aligned}
\end{align*}
\end{proof}

The next lemma gives the bound of $\rho_n$. 

\begin{Lemma}\label{lem:cs-02}
Suppose that $u_0\in H^2\II\cap H_0^1\II$  and $|u_0| \le 1$.
Let $v_n$ be the solution of   \eqref{eqn:vn-1}  and $u_n$ be  the solution of  the CS scheme \eqref{eqn:cs}. Then $\rho_n = v_n-u_n$  satisfies
\begin{equation}\label{error-estimate}
\max_{1\le n\le N}  \|\rho_n\|_{L^2(\Omega)}\le c\tau^\alpha .
\end{equation}
where the constant $c$ may depends on $\alpha, \kappa, T, u_0$ but is independent of $u$ and $\tau$.
\end{Lemma}
\begin{proof}
We note that $\rho_n$ satisfies the following discrete problem
\begin{equation*}
\bar\partial_\tau^\alpha \rho_n + A \rho_n
=(u(t_n) - u_{n-1}) + (u(t_n)^3-(u_n)^3) ,\qquad \text{with} ~~\rho_0 = 0.
\end{equation*}
By applying the discrete maximal $\ell^p$-regularity in Lemma \ref{lem:max-reg}, we obtain that for all $1<p<\infty$:
\begin{equation*}
  \begin{aligned}
    \|(\bar\partial_\tau^\alpha \rho_n)_{n=1}^m\|_{\ell^p(L^2(\Omega))}
&\le c\|(u(t_n) - u_{n-1})_{n=1}^m\|_{\ell^p(L^2(\Omega))}  + c\|(u(t_n)^3-(u_n)^3)_{n=1}^m\|_{\ell^p(L^2(\Omega))} = \sum_{i=1}^2 I_i.
  \end{aligned}
\end{equation*}
Using the regularity estimate \eqref{reg-PDE3}, we have an estimate for $I_1$
\begin{equation*}
  \begin{aligned}
 I_1 &\le c\|(u(t_{n-1}) - u_{n-1})_{n=1}^m\|_{\ell^p(L^2(\Omega))} + c\|(u(t_{n}) - u(t_{n-1}))_{n=1}^m\|_{\ell^p(L^2(\Omega))} \\
 &\le c\tau^\alpha + c\|(\rho_n)_{n=1}^m\|_{\ell^p(L^2(\Omega))} + c\Big|\Big|\Big(\int_{t_{n-1}}^{t_n} u'(s)\,ds \Big)_{n=1}^m\Big|\Big|_{\ell^p(L^2(\Omega))} \\
 &\le c\tau^\alpha + c\|(\rho_n)_{n=1}^m\|_{\ell^p(L^2(\Omega))}.
  \end{aligned}
\end{equation*}
Finally, using the fact that $|u(x,t)|\le 1$ and $|u^n|\le 1$ respectively by Theorems \ref{thm:max} and \ref{thm:dmp-cs}, we derive the bound for the second term
\begin{equation*}
  \begin{aligned}
  I_2 \le c \|(u(t_n) - u_n )_{n=1}^m\|_{\ell^p(L^2(\Omega))} &\le c\|(\theta_n)_{n=1}^m\|_{\ell^p(L^2(\Omega))}+ c\|(\rho_n)_{n=1}^m\|_{\ell^p(L^2(\Omega))}\\
  &\le c\tau^\alpha + c\|(\rho_n)_{n=1}^m\|_{\ell^p(L^2(\Omega))}.
  \end{aligned}
\end{equation*}
Combining the preceding three estimates, we arrive at
\begin{equation*}
  \begin{aligned}
    \|(\bar\partial_\tau^\alpha \rho_n)_{n=1}^m\|_{\ell^p(L^2(\Omega))} \le  c\tau^\alpha + c\|(\rho_n)_{n=1}^m\|_{\ell^p(L^2(\Omega))}.
  \end{aligned}
\end{equation*}
By choosing $p>1/\alpha$ and applying the discrete Gr\"onwall's inequality in Lemma \ref{lem:disc-Gronwall}), we obtain
\begin{align*}
\max_{1\le n\le N}\|\rho_n\|_{L^2(\Omega)}\le c\tau^\alpha .
\end{align*}
This completes the proof of the lemma.
\end{proof}

Combing Lemmas  \ref{lem:cs-01} and \ref{lem:cs-02}, we have the error estimate for the CS scheme \eqref{eqn:cs}.

\begin{theorem}\label{thm:error-cs}
Suppose that $u_0\in H^2\II\cap H_0^1\II$  and $|u_0| \le 1$.  Let $u$ be the solution of the time-fractional Allen-Cahn equation \eqref{eqn:phase-field}.
Then $u_n$, the solution of the CS scheme \eqref{eqn:cs}, satisfies the error estimate
\begin{equation}\label{error-estimate-cs}
\max_{1\le n\le N}  \|u (t_n)-u_n\|_{L^2(\Omega)}\le c\tau^\alpha .
\end{equation}
where the constant $c$ may depends on $\alpha, \kappa, T, u_0$ but is independent of $u$ and $\tau$.
\end{theorem}

\subsection{Error analysis of the weighted time stepping scheme.}
Now we shall turn to the error estimate for the WCS scheme \eqref{eqn:wcs} and the LWS scheme \eqref{eqn:lws}.
Besides the useful tools applied in the previous section, we also need the following lemma
for the bound of $\bar\partial_\tau^\alpha u(t_n)$.

\begin{Lemma}\label{lem:partu}
Let $u_0\in H^2\II\cap H_0^1\II$ and $u$ be the solution of the time-fractional Allen-Cahn equation \eqref{eqn:phase-field}.
Meanwhile, let $\bar\partial_\tau^\alpha$  denote the BE-CQ given by \eqref{eqn:BE-CQ}.
Then we have for $n\ge 1$
\begin{equation*}
  \max_{1 \le n\le N} \|\bar\partial_\tau^\alpha u(t_n)\|_{L^2(\Omega)}\le c .
\end{equation*}
\end{Lemma}

\begin{proof}
By setting $y(t) = u(t)-u_0$, then we have
\begin{equation*}
 \bar\partial_\tau^\alpha u(t_n) = \bar\partial_\tau^\alpha y(t_n) = \bar\partial_\tau^{\alpha-1} \psi_n = \sum_{j=0}^n \omega_{n-j} \psi_j,
\end{equation*}
where $\psi_n =  \bar\partial_\tau^{1} y(t_n) $, for $n=1,\ldots,N$ and $\psi_0=0$. By the regularity estimate \eqref{reg-PDE3}, we have
\begin{equation*}
  \| \psi_1 \|_{L^2\II} \le \tau^{-1}\|  \int_0^\tau u'(s) \,\d s  \|_{L^2\II} \le   \tau^{-1}\int_0^\tau \| u'(s) \|_{L^2\II} \,\d s
  \le c  \tau^{-1}\int_0^\tau s^{\alpha-1} \,\d s \le c\tau^{\alpha-1}.
\end{equation*}
Meanwhile, for $n\ge2$, we derive that
\begin{equation*}
  \| \psi_n \|_{L^2\II} \le \tau^{-1}\|  \int_{t_{n-1}}^{t_n} u'(s) \,\d s  \|_{L^2\II} \le   \tau^{-1}\int_{t_{n-1}}^{t_n} \| u'(s) \|_{L^2\II} \,\d s
  \le c  \tau^{-1}\int_{t_{n-1}}^{t_n} s^{\alpha-1} \,\d s \le c\tau^{-1} t_{n-1}^{\alpha-1}.
\end{equation*}
Finally, find a bound for the weights $\omega_{n}^{(\alpha-1)} = \Pi_{i=1}^n (1-\frac\alpha{j})$ in \eqref{eqn:alpha-1}.
By the trivial inequality $\ln(1 + x) \le  x$ for $x > -1$, we derive
\begin{equation*}
\begin{aligned}
\ln\omega_{n}^{(\alpha-1)} =  \sum_{j=1}^n \ln\bigg(1 - \frac{\alpha}{j}\bigg)& \leq -\alpha \sum_{j=1}^nj^{-1}
\le  -\alpha \ln (n+1) .
  \end{aligned}
\end{equation*}
which indicates that  for $n\ge 0$
 satisfy the estimate that
\begin{equation*}
 0<\omega_{n}^{(\alpha-1)} < (n+1)^{-\alpha}.
\end{equation*}
Therefore,
\begin{equation*}
  \begin{aligned}
    \|\bar\partial_\tau^\alpha u(t_n)\|_{ L^2(\Omega)} &\le   \tau^{1-\alpha} \sum_{j=2}^n \omega_{n-j}^{(\alpha-1)} \| \psi_j \|_{L^2\II} +  \tau^{1-\alpha} \omega_{n-1}^{(\alpha-1)}\| \psi_1 \|_{L^2\II}\\
    & \le   c \sum_{j=2}^n (n-j+1)^{-\alpha} (j-1)^{\alpha-1} +  c n^{-\alpha} \le c. \\
  \end{aligned}
\end{equation*}
where the constant $c$ is independent of $n$. This completes the proof.
\end{proof}

\begin{theorem}\label{thm:error-wcs}
Suppose that $u_0\in H^2\II\cap H_0^1\II$  and $|u_0| \le 1$. Let $u$ be the solution of the time-fractional Allen-Cahn equation \eqref{eqn:phase-field}.
Then $u_n$, the solution of the WCS scheme \eqref{eqn:wcs}, satisfies
\begin{equation*}
\max_{1\le n\le N}  \|u(t_n)-u_n\|_{L^2(\Omega)}\le c\tau^\alpha .
\end{equation*}
where the constant $c$ may depends on $\alpha, \kappa, T, u_0$ but is independent of $u$ and $\tau$.
\end{theorem}
\begin{proof}
To derive the error, we shall use the splitting \eqref{eqn:error-split} and note that the estimate for $\theta_n$ has been given in the Lemma \eqref{lem:cs-01}.
Then $\rho_n$ satisfies the time stepping problem
\begin{equation*}
\bar\partial_\tau^\alpha \rho_n + A \rho_n
=(u(t_n) - u_{n,\alpha}) + (u(t_n)^3-(u_n)^3) ,\qquad \text{with} ~~\rho_0 = 0.
\end{equation*}
Using the discrete maximal $\ell^p$-regularity in Lemma \ref{lem:max-reg}, we obtain that for all $1<p<\infty$:
\begin{equation*}
  \begin{aligned}
    \|(\bar\partial_\tau^\alpha \rho_n)_{n=1}^m\|_{\ell^p(L^2(\Omega))}
&\le  c\|(u(t_n) - u_{n,\alpha})_{n=1}^m\|_{\ell^p(L^2(\Omega))} + c\|(u(t_n)^3-(u_n)^3)_{n=1}^m\|_{\ell^p(L^2(\Omega))} =:   I_1+I_2.
  \end{aligned}
\end{equation*}
The second term could be bounded using the same argument in Lemma \ref{lem:cs-02}, with the help of the (discrete) maximum principle
and regularity estimate \eqref{reg-PDE3}. In particular, we have
\begin{equation*}
  \begin{aligned}
  I_2 \le  c\tau^\alpha + c\|(\rho_n)_{n=1}^m\|_{\ell^p(L^2(\Omega))}.
  \end{aligned}
\end{equation*}
Now we turn to the first term.
Using the definition of fractional weighted extrapolation \eqref{eqn:frac-interp}, we have
\begin{equation*}
  \begin{aligned}
 I_1 \le  \tau^\alpha \|(\bar \partial_\tau^\alpha u(t_n))_{n=1}^m\|_{\ell^p(L^2(\Omega))} +
 {\Big|\hskip-1.5pt\Big|}\Big(\sum_{i=1}^{n-1} \omega_{n-i}(u_i - u(t_i))\Big)_{n=1}^m {\Big|\hskip-1.5pt\Big|}_{\ell^p(L^2(\Omega))} .
  \end{aligned}
\end{equation*}
Using the Young's inequality of discrete convolution, we arrive at
 \begin{equation*}
  \begin{aligned}
 I_1  \le  \tau^\alpha \|(\bar \partial_\tau^\alpha u(t_n))_{n=1}^m\|_{\ell^p(L^2(\Omega))} + c\|(u_n - u(t_n))_{n=1}^m\|_{\ell^p(L^2(\Omega))}
  \end{aligned}
\end{equation*}
Then Lemma \ref{lem:partu} leads to
 \begin{equation}\label{eqn:esti1}
  \begin{aligned}
 I_1  \le  c \tau^\alpha + c\|(u_n - u(t_n))_{n=1}^m\|_{\ell^p(L^2(\Omega))} \le   c \tau^\alpha + c\|(\rho_n)_{n=1}^m\|_{\ell^p(L^2(\Omega))}
  \end{aligned}
\end{equation}
Combining the preceding two estimates, we arrive at
\begin{equation*}
  \begin{aligned}
    \|(\bar\partial_\tau^\alpha \rho_n)_{n=1}^m\|_{\ell^p(L^2(\Omega))} \le  c\tau^\alpha + c\|(\rho_n)_{n=1}^m\|_{\ell^p(L^2(\Omega))}.
  \end{aligned}
\end{equation*}
By choosing $p>1/\alpha$ and applying the discrete Gr\"onwall's inequality in Lemma \ref{lem:disc-Gronwall}, we obtain
\begin{align*}
\max_{1\le n\le N}\|\rho_n\|_{L^2(\Omega)}\le c\tau^\alpha .
\end{align*}
This completes the proof of the lemma.
\end{proof}

Similar argument may derive the error estimate for the LWS scheme \eqref{eqn:lws} with smooth initial data.
\begin{Corollary}\label{cor:error-lws}
Suppose that $u_0\in H^2\II\cap H_0^1\II$ and $|u_0| \le 1$. Let $u$ be the solution of the time-fractional Allen-Cahn equation \eqref{eqn:phase-field}.
Then $\{u_n\}$, the solution of the LWS scheme \eqref{eqn:lws}, satisfies the error estimate
\begin{equation}\label{error-estimate-lws}
\max_{1\le n\le N}  \|u (t_n)-u_n\|_{L^2(\Omega)}\le c\tau^\alpha .
\end{equation}
where the constant $c$ may depends on $\alpha, \kappa, T, u_0$ but is independent of $u$ and $\tau$.
\end{Corollary}

\begin{proof}
To derive the error, we shall use the usual splitting \eqref{eqn:error-split} and note that the estimate for $\theta_n$ has been given in the Lemma \eqref{lem:cs-01}.
Then $\rho_n$ satisfies the time stepping problem
\begin{equation*}
\bar\partial_\tau^\alpha \rho_n + A \rho_n
= (u(t_n) - u_{n,\alpha}) + (u(t_n)^3-(u_{n,\alpha})^3) + S(u_n - u_{n-\alpha})
\end{equation*}
with $\rho_0 = 0$.
Using the discrete maximal $\ell^p$-regularity in Lemma \ref{lem:max-reg}, as well as the fact that $|u(x,t)| \le 1 $ and $|u_{n,\alpha}(x)| \le 1$ for all $n\ge 0$, we obtain that for all $1<p<\infty$:
\begin{equation*}
  \begin{aligned}
    \|(\bar\partial_\tau^\alpha \rho_n)_{n=1}^m\|_{\ell^p(L^2(\Omega))}
&\le  c\|(u(t_n) - u_{n,\alpha})_{n=1}^m\|_{\ell^p(L^2(\Omega))}   + \| (u_n - u_{n-\alpha})_{n=1}^m\|_{\ell^p(L^2(\Omega))} =:I_1+I_2.
  \end{aligned}
\end{equation*}
The estimate for the first term has been given by \eqref{eqn:esti1}. Now we apply Lemma \ref{lem:partu} to obtain that
\begin{equation*}
  \begin{aligned}
I_2 &= c\| (\bar\partial_\tau^\alpha u_n)_{n=1}^m  \|_{\ell^p(L^2(\Omega))} \\
&\le   c \tau^\alpha \|(\bar\partial_\tau^\alpha u(t_n) )_{n=1}^m\|_{\ell^p(L^2(\Omega))}+
c \tau^\alpha \|(\bar\partial_\tau^\alpha (u(t_n)-u_n))_{n=1}^m\|_{\ell^p(L^2(\Omega))} \\
 & \le c \tau^\alpha + c \tau^\alpha \|(\bar\partial_\tau^\alpha (u(t_n)-u_n))_{n=1}^m\|_{\ell^p(L^2(\Omega))}.
  \end{aligned}
\end{equation*}
Then the discrete (in time) inverse inequality yields that
 \begin{equation*}
  \begin{aligned}
I_2 \le c \tau^\alpha + c  \|( u(t_n)-u_n)_{n=1}^m\|_{\ell^p(L^2(\Omega))} \le c \tau^\alpha + c\|(\rho_n)_{n=1}^m\|_{\ell^p(L^2(\Omega))}.
  \end{aligned}
\end{equation*}
Finally, the preceding estimates together with the discrete Gr\"onwall's inequality in Lemma \ref{lem:disc-Gronwall} result in the desired result.
\end{proof}

\begin{remark}
In this paper, we only present the argument for the homogeneous Dirichlet boundary condition.
In fact, all the argument could be applied to other types of boundary conditions,
e.g., periodic boundary condition and Nuemann boundary condition, since the analysis only depends on
the abstract operator $A$ and its resolvent estimate.
\end{remark}\bigskip

\section{Numerical results}\label{sec:numerics}
In this section, we present some numerical experiments to confirm the theoretical findings and to offer
new insight on the time-fractional Allen-Cahn dynamics.\vskip5pt
\begin{example}\label{Example 5.1.}
In the first example, we use a two-dimensional problem to support the convergence rate of the proposed numerical schemes.
To this end, we let $\Omega=(0,1)^2$ and $\kappa=0.5$, and take the fractional power $\alpha$ to be $0.4$, $0.6$, $0.8$, respectively.
We use the smooth initial condition $u_0(x,y)=x(1-x)y(1-y)$.
\end{example}

In the computation, The central finite difference method is used for the discretization in spatial space  with $h=1/200$ in each direction.
Here we present the error of numerical solution of the CS \eqref{eqn:cs}, WCS \eqref{eqn:wcs} and LWS schemes \eqref{eqn:lws},
in Table \ref{tab:a-time}. Since the exact solution is unknown, to numerically evaluate  pointwise-in-time temporal error $e_t$, we compute
\begin{equation*}
  e_t \approx \max_{1\le n\le N}\|u_\tau^n-u_{\tau/2}^{2n}\|_{L^2(\Omega)}.
\end{equation*}
Numerical experiments show that all the three schemes are convergent with rate  $O(\tau^\alpha)$.
This observation fully supports our theoretical result in Theorems \ref{thm:error-cs} and \ref{thm:error-wcs}, and
Corollary \ref{cor:error-lws}.

\begin{table}[htb!]
\caption{Example \ref{Example 5.1.}: $e_t$
with $T= 1$,  $\tau=T/N$,  $N=k\times10^4$, and $h=1/200$.}\label{tab:a-time}
\begin{center}
\vspace{-.3cm}{\setlength{\tabcolsep}{7pt}
     \begin{tabular}{|c|c|ccccc|c|}
     \hline
     Scheme & $\alpha\backslash k$    &$1 $ &$2 $ & $4 $ & $8 $ &$16$  &rate \\
          \hline
       &  $0.4$             &6.01e-4  &4.50e-4  &3.36e-4  &2.51e-4 &1.87e-4 & $\approx$ 0.42 (0.40)\\
     CS  &  $0.6$           &1.38e-4  &8.78e-5  &5.64e-5  &3.65e-5 &2.37e-5 & $\approx$ 0.63 (0.60)\\
      &  $0.8$              &2.51e-5  &1.26e-5 &6.51e-6 &3.62e-6 &2.04e-6 &  $\approx$ 0.84 (0.80)\\
      \hline
       &  $0.4$           &4.84e-4 &3.76e-4 &2.90e-4 &2.22e-4 &1.70e-4 & $\approx$ 0.38 (0.40)\\
     WCS  &  $0.6$           &1.32e-4 &8.54e-5  &5.54e-5  &3.60e-5 &2.35e-5 & $\approx$ 0.62 (0.60)\\
      &  $0.8$              &3.32e-5  &1.79e-5 &9.72e-6 &5.29e-6 &2.89e-6 &  $\approx$ 0.85 (0.80)\\
      \hline
       &  $0.4$           &5.70e-4 &4.35e-4 &3.29e-4 &2.48e-4 &1.86e-4 & $\approx$ 0.41 (0.40)\\
     LWS  &  $0.6$           &1.70e-4 &1.10e-4  &7.16e-5  &4.67e-5 &3.05e-5 & $\approx$ 0.62 (0.60)\\
      &  $0.8$              &6.71e-5  &3.76e-6 &2.10e-6 &1.18e-6 &6.64e-6 &  $\approx$ 0.83 (0.80)\\
      \hline
     \end{tabular}}
\end{center}
\end{table}

\begin{example}\label{Example 5.2.} Consider the one-dimensional fractional Allen--Cahn equation in $[0,2\pi]$ with zero Dirichlet boundary conditions.
We fix $\kappa=0.1$ and take the fractional power $\alpha$ to be $0.5$, $0.7$, $0.9$, respectively.
The central finite difference method is used for the discretization in spatial space  with $N=2^7$, and the time step $\tau=10^{-2}$. 
Here we test smooth initial value: $u(0,x)=0.05\sin(x)$.
\end{example}
\begin{figure}[!h]
\centering
\includegraphics[width=0.8\textwidth,height=7cm]{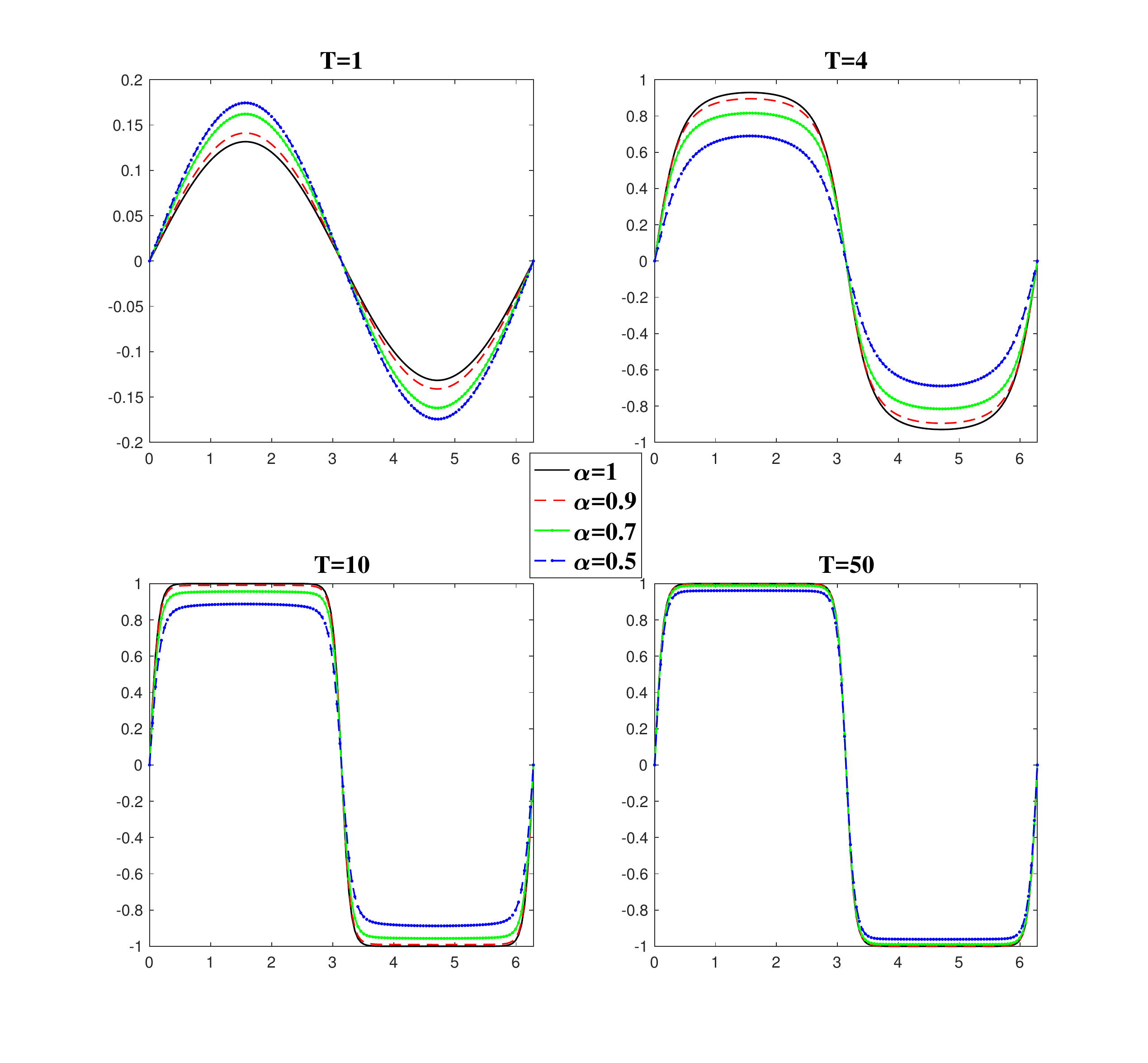}\\
\includegraphics[width=0.8\textwidth,height=7.5cm]{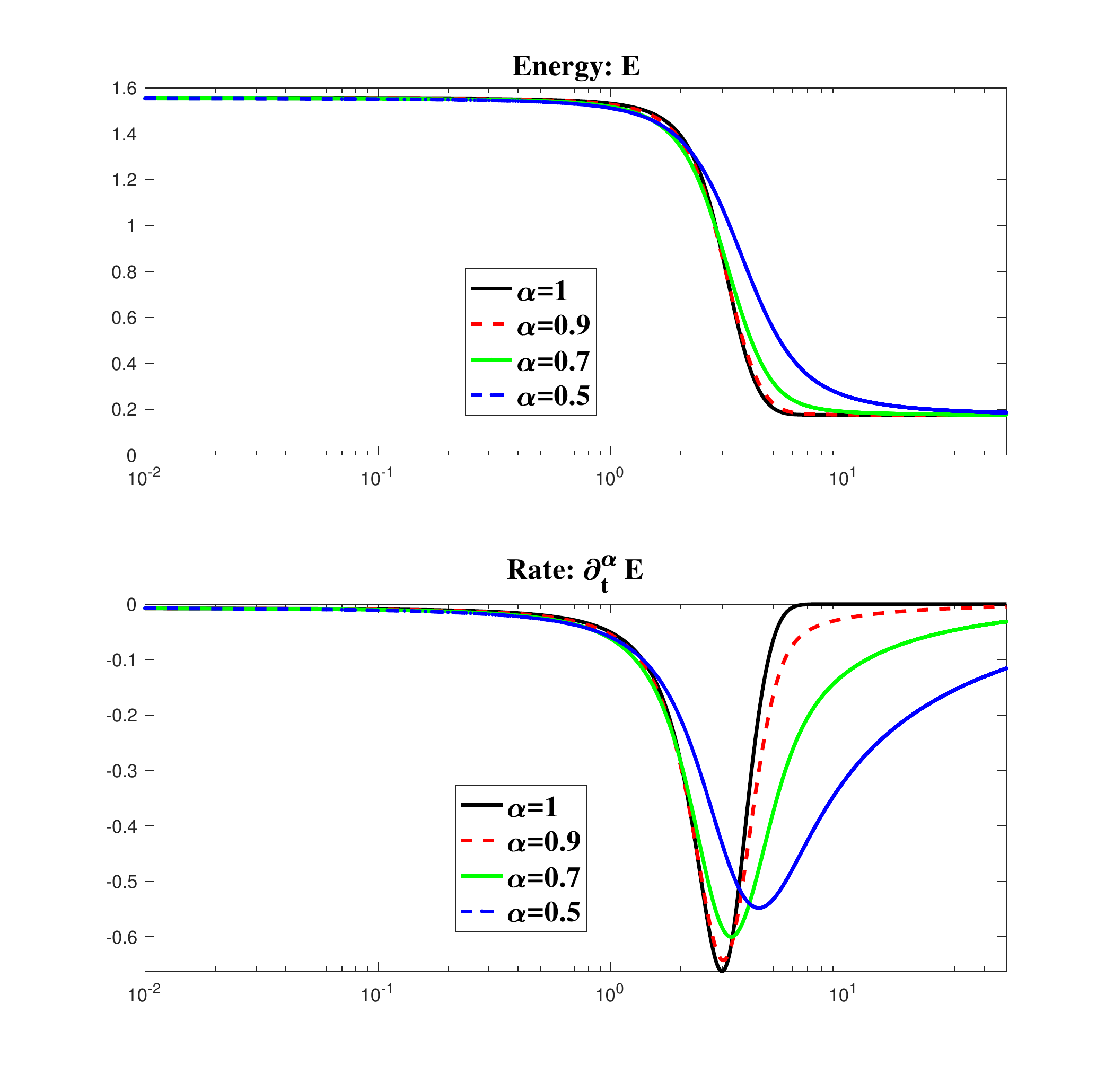}
\caption{(Example \ref{Example 5.2.}:) Above four subfigures: configurations of phase field evolution; below two subfigures: evolutions of energy and fractional derivative of energy.}\label{ex2:1}
\end{figure}

The classical Allen--Cahn equation ($\alpha=1$) is also computed for comparison.
We present the solutions and energy curves in Figure \ref{ex2:1}.
We observe that for time fractional case, the evolution is slower than that of the classical Allen--Cahn.
In particular,
we observe that the dynamic is slower for the smaller $\alpha$, especially for long time, e.g. after $T=10$, where the classical Allen--Cahn
almost arrives at the steady state but the fractional one with $\alpha=0.5$ is far from the steady state.
But in both two cases, the energy decays monotonically, and moreover, the following energy dispassion law
$$\bar{\partial}^\alpha_\tau E(t_n)\le 0,\qquad \text{for all}~~n \ge 1,$$
holds as expected.

\begin{example}\label{Example 5.3.}
Consider the two-dimensional fractional Allen--Cahn equation in $(0,2\pi)^2$ with zero Dirichlet boundary conditions.
We fix $\kappa=0.1$ and take the fractional power $\alpha$ to be $0.5$, $0.7$, $9$, respectively.
The central finite difference method is used for the discretization in spatial space  with $N=2^7$ in each direction, and the time step $\tau=5\times10^{-3}$.
In the experiment, we test the random initial value with zero mean
$u(0,x,y)=0.05(2*\text{rand}(x,y)-1),$
where $\text{rand}(\cdot)$ denotes a random number in $[0,1]$.
\end{example}

\begin{figure}[!h]
\centering
\includegraphics[width=0.8\textwidth,height=16cm]{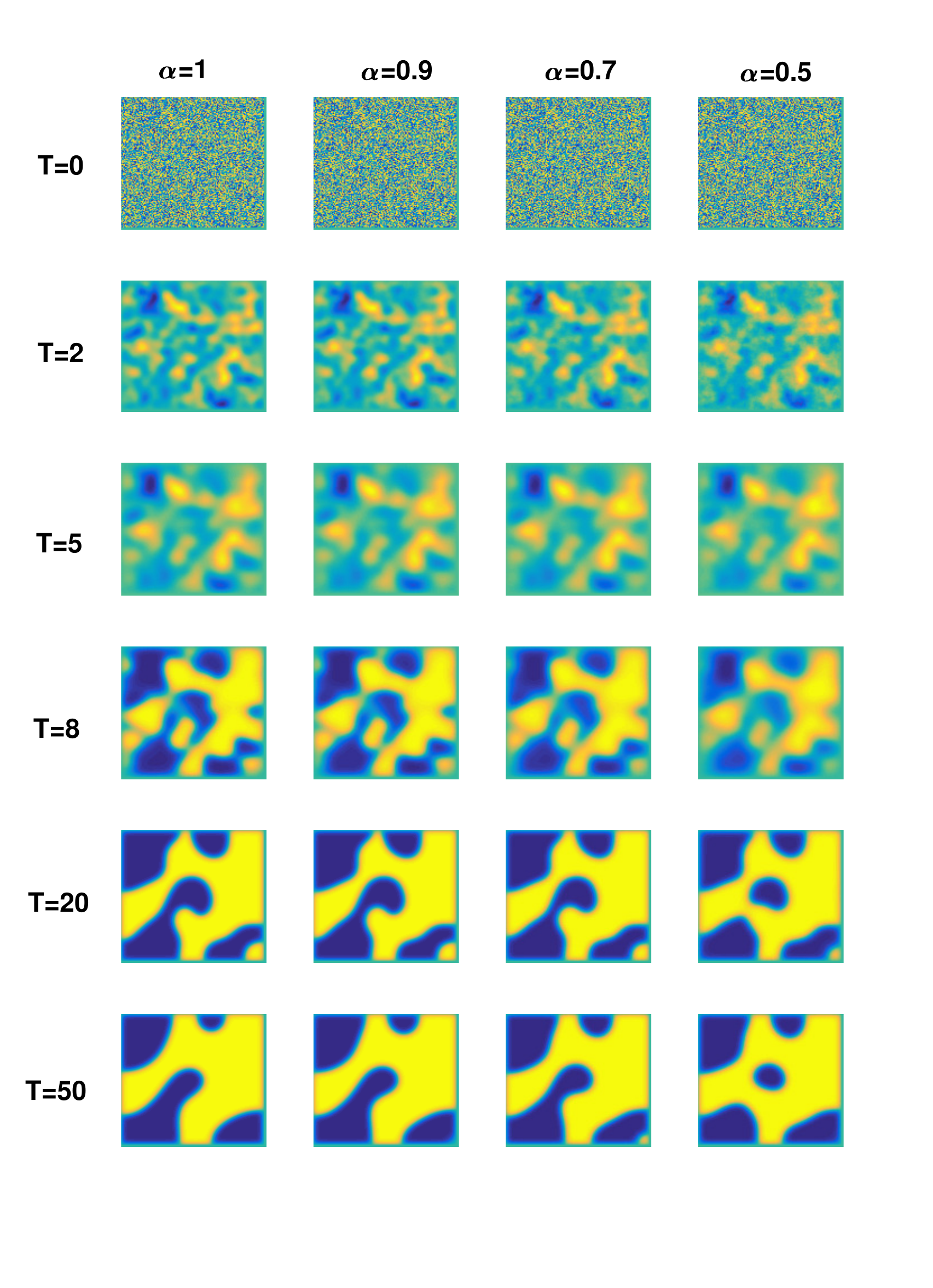}\vspace{-1cm}
\caption{(Example \ref{Example 5.3.}) Configurations of phase field evolution.}\label{ex3:1}
\end{figure}
\begin{figure}[!h]
\centering
\includegraphics[width=0.8\textwidth,height=7.5cm]{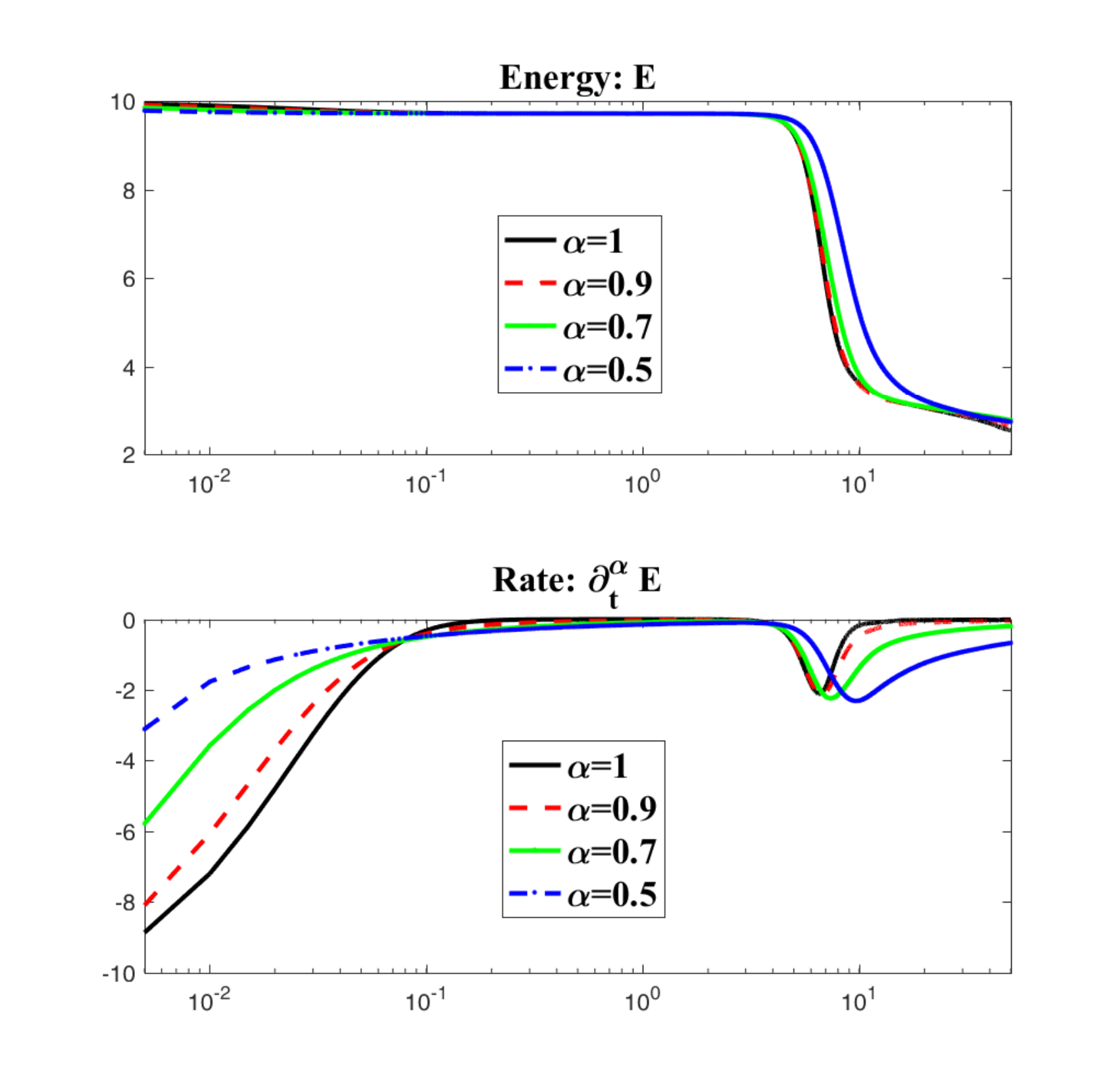}
\caption{(Example \ref{Example 5.3.}) Evolutions of energy and fractional derivative of energy}\label{ex3:2}
\end{figure}

In Figure \ref{ex3:1}, the numerical solutions look almost the same before $T=5$ for all four equations,
which corresponds to the phase separation period.  After that the solutions looking apparently different,
 in this period the dynamics enters the long time scale phase coarsening. This is further verified by the energy evolution.
 Numerically, we can conclude that the equations with different fractional orders have similar dynamics for the phase separation
 but the coarsening dynamics becomes slower for the smaller  fractional orders.
Again, from the numerical experiments, we observe the fractional energy dispassion law
$$\bar{\partial}^\alpha_\tau E(u_n)\le 0,\qquad \text{for all}~~n \ge 1.$$

\begin{example}\label{Example 5.4.}
It well-known that the classical Allen--Cahn equation will converge to the mean curvature flow.
In this example, we intend to numerically compare the difference between the dynamics of the
integer order AC equation and the ones of the time-fractional Allen--Cahn equation.
To this end, we continue the two-dimensional problem with $\kappa=0.1$. 
We shall test two experiments with  initial data representing different interfaces, i.e., a circle (I) and a dumbbell (II),
and observe how the phase parameter and the energy evolve for different values of $\alpha$.
\end{example}

The results for the first experiment are presented in Figure \ref{ex4:1}, and those for the second one are shown in Figure \ref{ex4:2}, respectively.
In both two tests, it is seen that the relaxation time changes with different $\alpha$.
For the circular case, after a short time for phase reordering, all three circles shrink as time going.
Particularly, the area of the circle in the classical Allen--Cahn equation shrinks linearly,
which essentially illustrate the dynamics of mean curvature flow in the two-dimension.
For the time-fractional Allen--Cahn model, the area shrinking rate seems to exhibit a power-law scaling.
The detailed behavior will not be presented here.
For the dumbbell case, in all three equations it is observed that two balls shrink first while the neck in between keeps flat,
where the curvature is zero. It is reasonable to conjecture that the dynamics of time-fractional AC equation may converge
to some kind of mean curvature flow, at least for the simple structure cases. This will be an interesting topic in our future studies.

\begin{figure}[!h]
\centering
\includegraphics[width=0.8\textwidth,height=7cm]{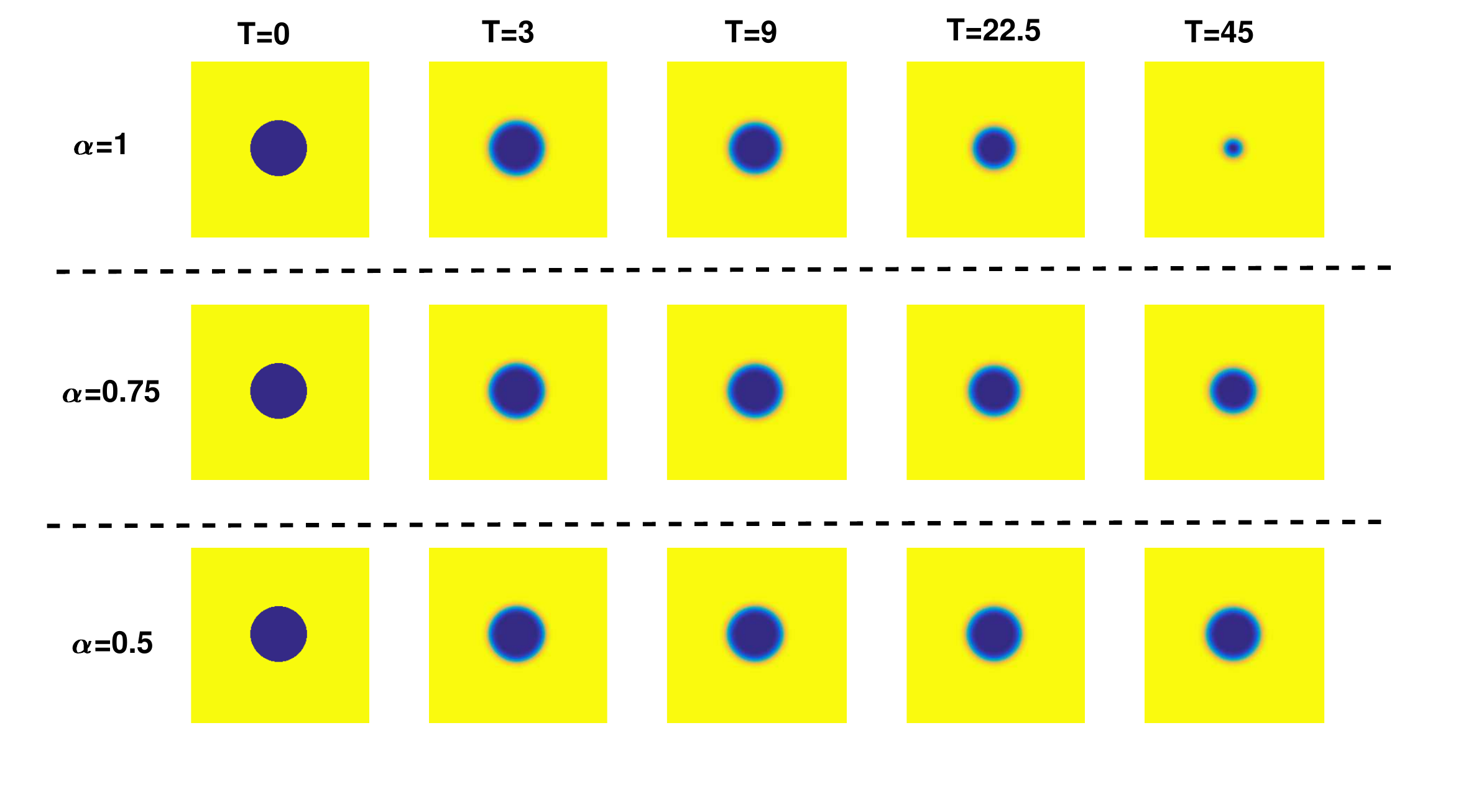}\vspace{0cm}\\
\includegraphics[width=0.8\textwidth,height=5.5cm]{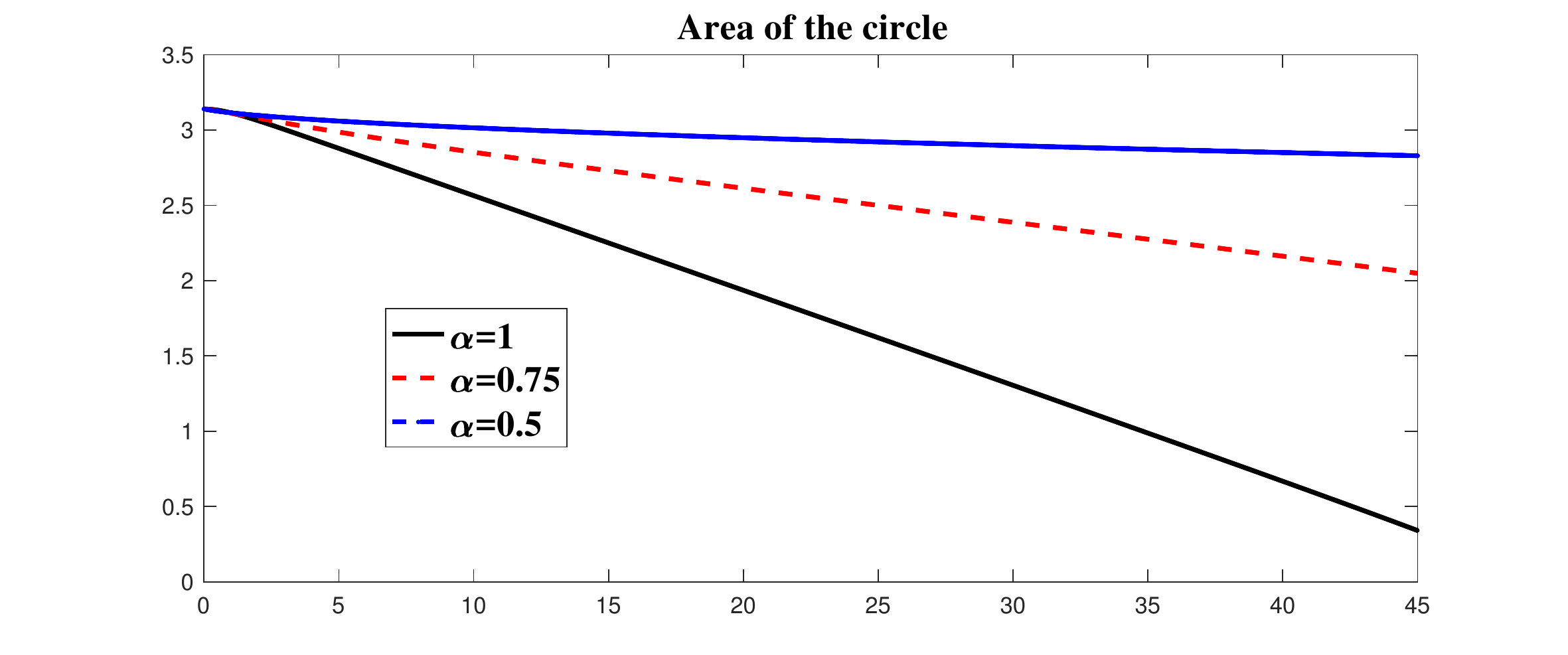}
\caption{(Example \ref{Example 5.4.} (I)) Comparisons between classical AC and factional AC.}\label{ex4:1}
\end{figure}

\begin{figure}[!h]
\centering
\includegraphics[width=0.8\textwidth,height=7cm]{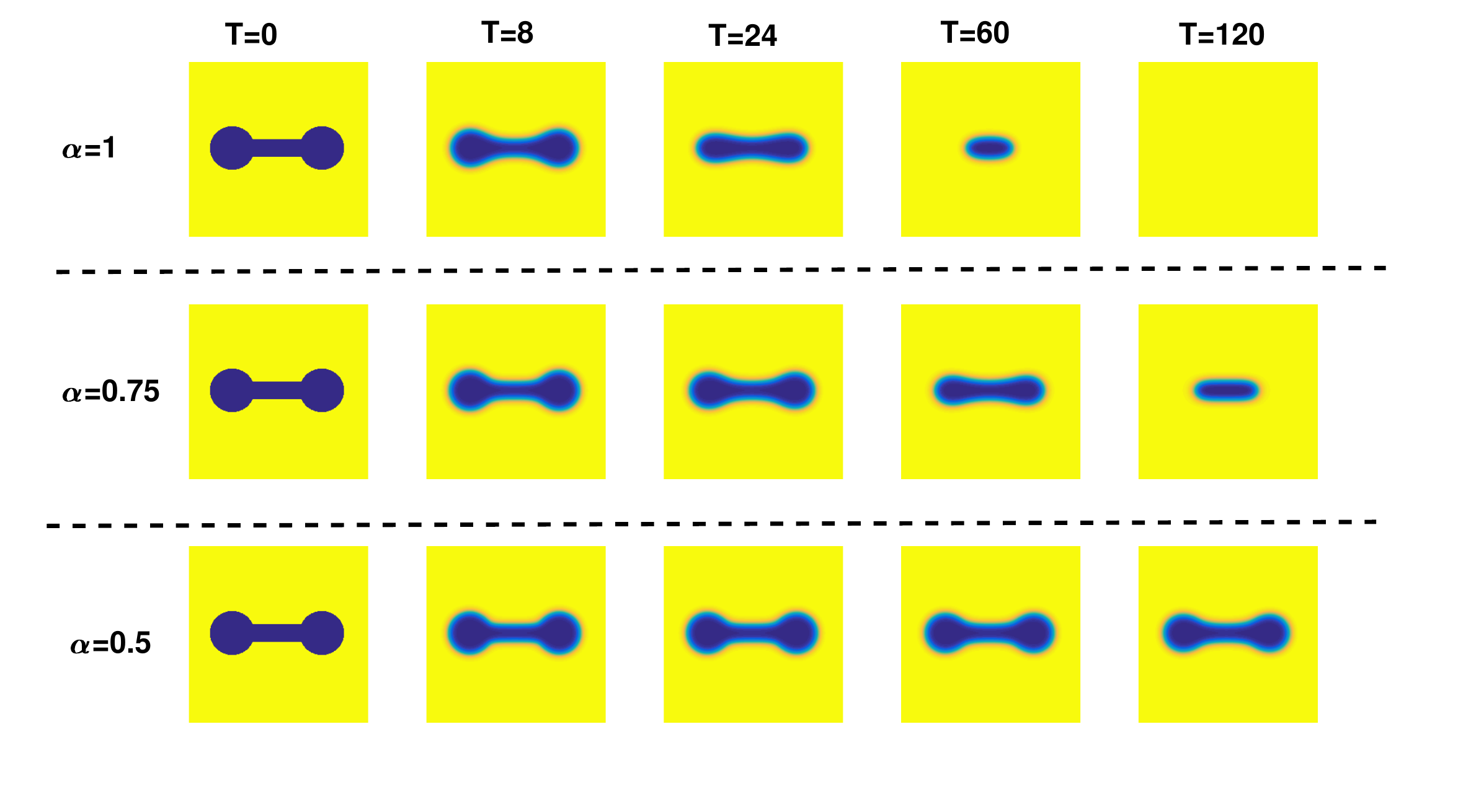}\vspace{0cm}\\
\includegraphics[width=0.8\textwidth,height=5.5cm]{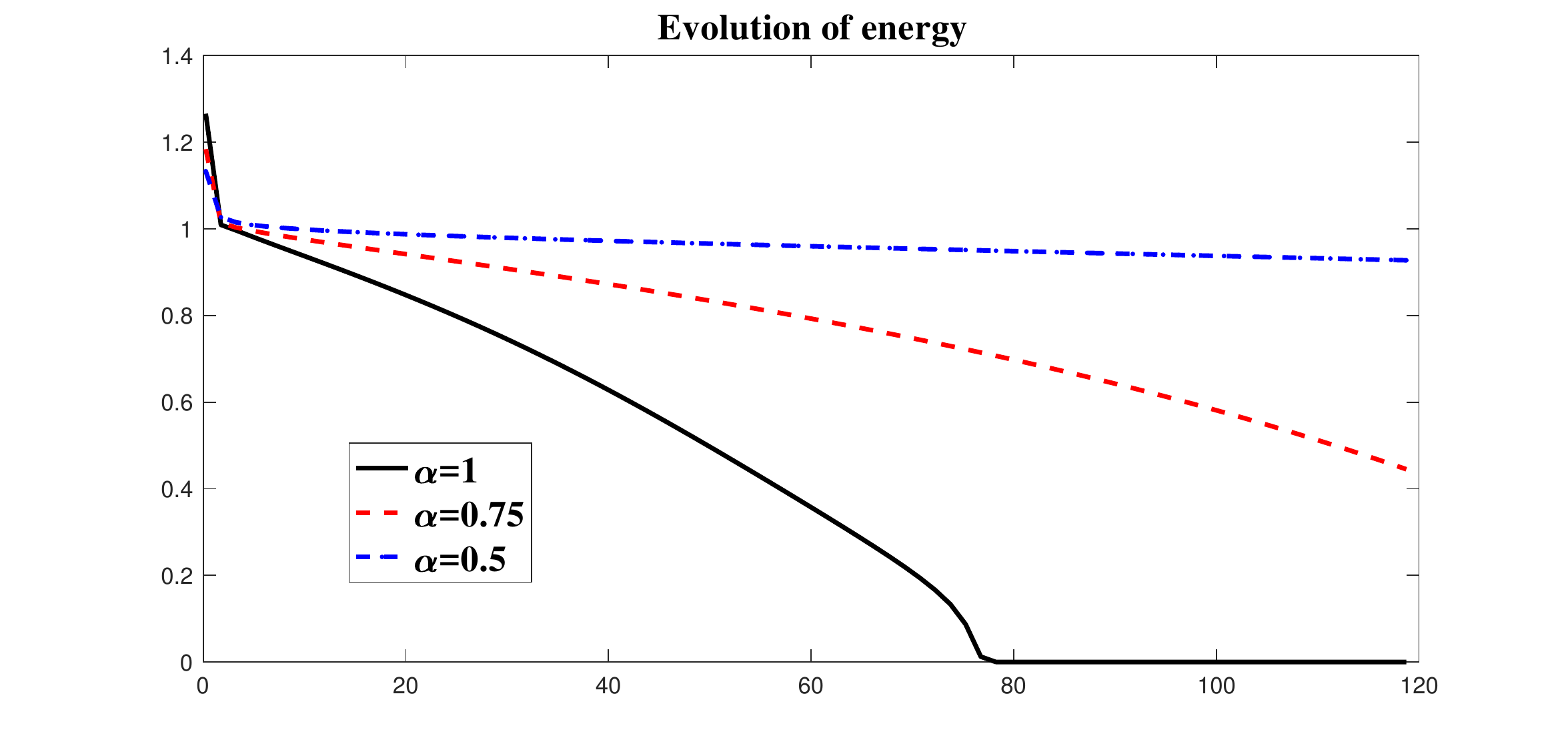}
\caption{(Example \ref{Example 5.4.} (II)) Comparisons between classical AC  and factional AC.}\label{ex4:2}
\end{figure}

\section{conclusion}
In this paper, we study a time-fractional Allen-Cahn equation, where the conventional first-order time-derivative is replaced by a fractional
derivative with order $\alpha\in(0,1)$, resulting in history dependent nonlocal-in-time dynamics.
The well-posedness, solution regularity, and maximum principle are proved, by using some useful tools such as the maximal
$L^p$ regularity of fractional evolution equations and the fractional Gr\"owall's inequality.
Then we developed three unconditionally solvable and stable time stepping schemes, i.e., convex splitting scheme, weighted convex splitting scheme and linear weighted stabilized scheme.
The energy dissipation property (in a weighted average sense) were discussed for the last two time stepping schemes.
Finally, we show the convergence of those time stepping schemes, where the error is of order $O(\tau^\alpha)$ without any extra regularity assumption on the solution.

It is promising to extend our argument to the general nonlocal-in-time phase field model
$$D_\rho u(t) -\kappa^2 \Delta u = -F'(u),$$
with historical initial data, where the nonlocal operator is given by
$$D_\rho u(t) = \int_0^\infty (u(t) - u(t-s))\rho(s)\,\d s,$$
which contains many popular examples, such as Caputo fractional derivative we discussed in this paper, tempered fractional derivative \cite{Sabzikar:2015},
distributed-order derivative \cite{Chechkin:2002}, or nonlocal operator with a finite nonlocal horizon \cite{DuYangZhou:2017},
even though the theoretical analysis of those models remain fairly scarce, and those technical tools, such as (discrete) maximal regularity
and weighted Gr\"onwall's inequality, await rigorous study.
Finally, the simulated dynamics of the nonlocal-in-time Allen-Cahn equations observed in numerical experiments reveal
interesting insight on nonlocal-in-time curvature dependent dynamics, which
awaits further theoretical studies in the future.

\bibliographystyle{abbrv}

\end{document}